\documentclass[12pt]{article}
\usepackage{latexsym,amssymb,amsmath,amsfonts,amsthm}

\usepackage{lineno}
\usepackage{color}
\usepackage{epsfig}
\usepackage{graphics}
\usepackage{enumerate}

\newtheorem{theorem}{Theorem}

\newtheorem{corollary}{Corollary}

\newtheorem{lemma}{Lemma}
\newtheorem*{corollary*}{Corollary}

\newtheorem{proposition}{Proposition}

\newcommand{\ve}{{\varepsilon}}

\newcommand{\Var}{\mathrm{Var}}

\newcommand{\pabbrev}{q_{j,+}^k(t-\Delta\ell)}

\newcommand{\pabbrevvv}{q_{j,+}^k(t-\Delta\ell_2)}

\usepackage{hyperref}

\newcommand{\E}{\mathbb{E}}

\def\beq{ \begin{equation} }
\def\eeq{ \end{equation} }

\def\square{\vcenter{\vbox{\hrule height .4pt
  \hbox{\vrule width .4pt height 5pt \kern 5pt
        \vrule width .4pt} \hrule height .4pt}}}

\def\sqz{\kern-0.2em}

\newcommand{\inquotes}[1]{``#1''}

\usepackage{comment}

\usepackage{graphicx}
\graphicspath{%
    {converted_graphics/}% inserted by PCTeX
    {/}% inserted by PCTeX
}

\title{
Limit theorems for the site frequency spectrum \\ of neutral mutations in an \\ exponentially growing population
}

\author{Einar Bjarki Gunnarsson$^{1,2,3}$ \and\hspace*{-6pt}  Kevin Leder$^{1}$ \and\hspace*{-6pt}Xuanming Zhang$^{1}$}
\date{%
    \footnotesize     $^1$Department of Industrial and Systems Engineering, University of Minnesota, Twin Cities, Minneapolis, MN, USA. \\[3pt]
    $^2$School of Mathematics, University of Minnesota, Twin Cities, Minneapolis, MN, USA. \\[2pt]
    $^3$Applied Mathematics Division, Science Institute, University of Iceland, Reykjavík, Iceland. \\[2pt]
}

\begin{document}

\maketitle
\begin{abstract}
    
The site frequency spectrum (SFS) is a widely used summary statistic of genomic data. Motivated by recent evidence for the role of neutral evolution in cancer, we investigate the SFS of neutral mutations in an exponentially growing population. Using branching process techniques, we establish (first-order) almost sure convergence results for the SFS of a Galton-Watson process, evaluated either at a fixed time or at the stochastic time at which the population first reaches a certain size. We finally use our results to construct consistent estimators for the extinction probability and the effective mutation rate of a birth-death process.
\\

\noindent {\bf Keywords:} Site frequency spectrum; Neutral evolution; Infinite sites model; Branching processes; Convergence of stochastic processes.\\

\noindent {\bf MSC2020 Classification:} 60J85, 60F15, 92D25, 92B05.
\end{abstract}
\section{Introduction}

The site frequency spectrum (SFS) is a popular summary statistic of genomic data,
recording the frequencies of mutations within a given population or population sample.
For the case of a large constant-sized population and selectively neutral mutations,
the expected value of the SFS has given rise to several estimators of the rate of mutation accumulation within the population,
and these estimators have formed the basis of many
statistical tests of neutral evolution vs.~evolution under selection
\cite{zeng2006statistical,achaz2009frequency}.
In this way, the SFS has provided a simple means
of understanding the rate and mode
of evolution 
in a population
using genomic data.

Motivated by the uncontrolled growth of cancer cell populations,
and the mounting evidence for the role of neutral evolution in cancer \cite{sottoriva2015big,ling2015extremely,williams2016identification,venkatesan2016tumor,davis2017tumor}, 
several authors have recently studied the SFS of neutral mutations
in an exponentially growing population.
Durrett \cite{durrett2013population,durrett2015branching} considered a supercritical birth-death process, in which cells live for an exponentially distributed time and then divide or die.
He showed that in the large-time limit, the expected number of mutations found at a frequency $\geq f$ amongst cells with infinite lineage follows a $1/f$ power law with $0<f<1$.
Similar results were obtained by Bozic et al.~\cite{bozic2016quantifying} and in a deterministic setting by Williams et al.~\cite{williams2016identification}.
In the aforementioned work, Durrett also derived an approximation for the expected SFS of
a small random sample taken from the population  \cite{durrett2013population,durrett2015branching}.
Further small sample results have been derived using both branching process and coalescence techniques and they have been compared with Durrett's result in
\cite{ohtsuki2017forward,dinh2020statistical}.
In \cite{gunnarsson2021exact}, we derived exact expressions for the SFS of neutral mutations in a supercritical birth-death process, both for cells with infinite lineage and for the total cell population, evaluated either at a fixed time (fixed-time SFS) or at the stochastic time at which the population first reaches a given size (fixed-size SFS). 
More recently, Morison et al.~analyzed the SFS, single-cell division distributions and mutational burden distributions in a supercritical birth-death process \cite{morison2023single}.
The effect of selective mutations on the expected SFS has been investigated by Tung and Durrett \cite{tung2021signatures} and Bonnet and Leman \cite{bonnet2023site}.
The latter work considers the setting of a drug-sensitive tumor which decays exponentially under treatment, with cells randomly acquiring resistance
which enables them to grow exponentially under treatment.

Whereas the aforementioned works have focused on the mean behavior of the SFS,
here,
we are interested in the asymptotic behavior of the underlying stochastic process.
Using the framework of coalescent point processes, Lambert \cite{lambert2009allelic} derived a strong law of large numbers for the SFS of neutral mutations in a population sample, ranked in such a way that coalescence times among consecutive individuals are i.i.d.
Later works by Lambert \cite{lambert2018coalescent}, Johnston \cite{johnston2019genealogy} and Harris et al.~\cite{harris2020coalescent} characterized the joint distribution of coalescence times for a uniformly drawn sample 
from a continuous-time Galton-Watson process. 
Building on these works, 
Johnson et al.~\cite{johnson2023estimating} 
derived limit distributions for the total lengths of 
internal and external branches 
in the genealogical tree of a birth-death process.
Schweinsberg and Shuai \cite{schweinsberg2023asymptotics} 
extended this analysis to  branches supporting exactly $k$ leaves,
which under a constant mutation rate characterizes the SFS of a uniformly drawn sample.
For a supercritical birth-death process,
the authors established both a weak law of large numbers and the asymptotic normality of branch lengths 
in the limit of a large sample,
assuming that the sample is sufficiently small compared to the
expected population size at the sampling time.

In this work, instead of considering a sample from the population using coalescence techniques,
we investigate the first-order asymptotics for the SFS of the total population using branching process techniques.
We establish results both for the fixed-time and fixed-size SFS
under the infinite sites model of mutation, where each new mutation is assumed to be unique \cite{durrett2008probability}.
Besides having theoretical value, our results can be applicable to the setting where an entire subclone of cells within a tumor is sampled, as opposed to cells being sampled randomly across the tumor \cite{gunnarsson2021exact}.
Our results can also potentially be applied to an {\em in vitro} setting where a single cell is expanded in 2D or 3D culture to a miniaturized version of a tumor.
Cheek and Antal recently studied a finite sites model in \cite{cheek2020genetic} (see also \cite{cheek2018mutation}),
where each genetic site is 
allowed to mutate back and forth between the four nucleotides $A,C,G,T$.
With the understanding that a site is mutated if its nucleotide differs from the nucleotide of the initial individual, the authors investigated the SFS of a birth-death process stopped at a certain size,
both for mutations observed in a certain number and in a certain fraction of individuals.
They used a limiting regime where the population size is sent to infinity, mutation rate is sent to 0, and the number of genetic sites is sent to infinity.
In contrast, we will assume a constant mutation rate under the infinite sites model (with no back mutations), 
and send either the fixed time or the fixed size at which the population is observed to infinity.

Our results are derived for a supercritical Galton-Watson process in continuous time, where each individual acquires neutral mutations at a constant rate $\nu>0$.
Let $Z(t)$ denote the size of the population at time $t \geq 0$,
$\lambda>0$ denote the net growth rate of the population,
$\tau_N$ denote the time at which the population first reaches size $N$,
and $S_j(t)$ denote the number of mutations found in $j \geq 1$ individuals at time $t$.
Our main result, Theorem \ref{thm:mainresult}, characterizes the first-order behavior of $e^{-\lambda t}S_j(t)$ as $t \to \infty$ (fixed-time result) and $N^{-1}S_j(\tau_N)$ as $N \to \infty$ (fixed-size result).
To prove the fixed-time result, the key idea is to decompose $(S_j(t))_{t \geq 0}$ into a difference of two increasing processes $(S_{j,+}(t))_{t \geq 0}$ and $(S_{j,-}(t))_{t \geq 0}$.
These processes count the total number of instances that a mutation reaches and leaves frequency $j$, respectively, up until time $t$.
Using the limiting behavior of $Z(t)$ as $t \to \infty$, we construct large-time approximations for the two processes $(S_{j,+}(t))_{t \geq 0}$ and $(S_{j,-}(t))_{t \geq 0}$.
We then establish exponential $L^1$ error bounds on these approximations,
which imply convergence in probability.
Finally, by adapting an argument of Harris (Theorem 21.1 of \cite{harris1964theory}), we use the exponential error bounds and the fact that $(S_{j,+}(t))_{t \geq 0}$ and $(S_{j,-}(t))_{t \geq 0}$ are increasing processes to show that $e^{-\lambda t}S_{j,+}(t)$ and $e^{-\lambda t} S_{j,-}(t)$ converge almost surely to their approximations. This in turn gives almost sure convergence for $e^{-\lambda t}S_j(t)$ as $t \to \infty$.
The fixed-size result is obtained by combining the fixed-time result with an approximation result for $\tau_N$, given by Proposition \ref{thm:passageApprox}.
Finally, we establish analogous fixed-time and fixed-size convergence results for $M(t) =  \sum_{j=1}^\infty S_j(t)$, the total number of mutations present at time $t$, in Proposition \ref{thm:totalnummutresult}.
All results are given conditional on nonextinction of the population.

The rest of the paper is organized as follows. Section 2 introduces our branching process model and establishes the relevant notation. Section 3 presents our results, including explicit expressions for the birth-death process. Section 4 outlines the proof of the main result, Theorem \ref{thm:mainresult}. 
Section 5 constructs consistent estimators for the extinction probability and effective mutation rate of the birth-death process. Finally, the proofs of the remaining results can be found in Section 6.

\section{Model}

\subsection{Branching process model with neutral mutations}

We consider a Galton-Watson branching process $(Z(t))_{t \geq 0}$, started with a single individual at time 0, $Z(0)=1$, where the lifetimes of individuals are exponentially distributed with mean $1/a>0$.
At the end of an individual's lifetime, it produces offspring according to the distribution
$(u_k)_{k \geq 0}$, where $u_k$ is the probability that $k$ offspring are produced.
We define $m := \sum_{k=0}^\infty ku_k$ as the mean number of offspring per death event
and assume that the offspring distribution has a finite third moment, 
$\sum_{k=0}^\infty k^3u_k<\infty$.
Each individual, over its lifetime,
accumulates neutral mutations at (exponential) rate $\nu>0$.
We assume the infinite sites model of mutation,
where each new mutation is assumed to be unique.
Throughout, we consider the case $m>1$ of a supercritical process.
The net growth rate of the population is then $\lambda=a(m-1)>0$, with $E[Z(t)]=e^{\lambda t}$ for $t \geq 0$. 

We will be primarily interested in analyzing the process conditional on long-term survival of the population.
We define the event of nonextinction of the population as
\begin{align*}
    \Omega_\infty := \{Z(t)>0 \;\text{for all $t > 0$}\}.
\end{align*}
We also define the probability of eventual extinction as 
\begin{align} \label{eq:extinctionprobdef}
    p := P(\Omega_\infty^c) = P(Z(t)=0 \;\text{for some $t>0$}\},
\end{align} 
and the corresponding survival probability as $q := P(\Omega_\infty)$.
For $N \geq 1$, we define $\tau_N$ as the time at which the population first reaches size $N$,
\begin{align} \label{eq:tau_N_def}
    \tau_N := \inf\{t\geq0: Z(t) \geq N\},
\end{align}
with the convention that $\inf \varnothing = \infty$.
Note that on $\Omega_\infty$,
$\tau_N<\infty$ almost surely.
Also note that if $u_k>0$ for some $k>2$, it is possible that $Z(\tau_N)>N$.
We finally define
\begin{align*}
    p_{i,j}(t) := P(Z(t)=j|Z(0)=i)
\end{align*}
as the probability of transitioning from $i$ to $j$ individuals in $t$ time units.
For the baseline case $Z(0)=1$, we simplify the notation to $p_j(t) := p_{1,j}(t)$.

\subsection{Special case: Birth-death process}

An important special case is that of the birth-death process,
where $u_2 > u_0 \geq 0$ and $u_0+u_2=1$.
In this process, an individual at the end of its lifetime either dies without producing offspring or produces two offspring.
At each death event, the population therefore either reduces or increases in size by one individual.
The birth-death process is for example relevant to the population dynamics of cancer cell populations (tumors) and bacteria.
In this case, the probability of eventual extinction can be computed explicitly as $p = u_0/u_2$ and the survival probability as $q = 1-u_0/u_2$ \cite{durrett2015branching}.
Furthermore, the probability mass function $j \mapsto p_j(t)$ has an explicit expression for each $t \geq 0$, given by expression \eqref{eq:sizedisgeneral} in Section \ref{app:bdsimplification}.
This will enable us to derive explicit limits for the site frequency spectrum of the birth-death process, see Corollary \ref{corollary} in Section \ref{sec:mainresultbd}.

\subsection{Asymptotic behavior} \label{sec:asymptoticbehavior}

We note that $(e^{-\lambda t}Z(t))_{t \geq 0}$ is a nonnegative martingale with respect to the natural filtration 
$\mathcal{F}_t := \sigma(Z(s); s\leq t)$.
Thus, 
there exists a random variable $Y$ such that $e^{-\lambda t}Z(t)\to Y$ almost surely $t\to\infty$.  
By Theorem 2 in Section III.7 of \cite{athreya2004branching},
\begin{align} \label{eq:Ydistribution}
    Y\stackrel{\mathcal{D}}{=}p\delta_0+q\xi,
\end{align}
where $p$ and $q$ are the extinction and survival probabilities of the population,
respectively,
$\delta_0$ is a point mass at 0,
and $\xi$ is a random variable on $(0,\infty)$ with a strictly positive continuous density function and mean $1/q$.
Since we assume that the offspring distribution has a finite third moment we know that $E[(Z(t))^2] = O(e^{2\lambda t})$ by Chapter III.4 of \cite{athreya2004branching} or Lemma 5 of \cite{foo2014escape}, hence
$(e^{-\lambda t}Z(t))_{t \geq 0}$ is uniformly integrable
and
$E[Y|\mathcal{F}_t]=e^{-\lambda t}Z(t)$. 

Based on the large-time approximation $Z(t) \approx Y e^{\lambda t}$, 
for $N \geq 1$, we define an approximation to the hitting time $\tau_N$ 
defined in \eqref{eq:tau_N_def} 
as follows:
\begin{align} \label{eq:tNdef}
    t_N := \inf\{t\geq0: Ye^{\lambda t} = N\},
\end{align}
with the understanding that $t_N = \infty$ if $Y=0$.
In Proposition \ref{thm:passageApprox}, we show that conditional on $\Omega_\infty$, $\tau_N-t_N \to 0$ almost surely as $N \to \infty$.

\subsection{Site frequency spectrum}

In the model, each individual accumulates neutral mutations at rate $\nu>0$.
For $t>0$, enumerate the mutations that occur up until time $t$ as $1,\ldots,N_t$,
and define ${\cal M}_t := \{1,\ldots,N_t\}$ as
the set of mutations generated up until time $t$.
For $i \in {\cal M}_t$ and $s \leq t$, let $C^i(s)$ denote the number of individuals at time $s$ that carry mutation $i$, with $C^i(s) = 0$ before mutation $i$ occurs.
The number of mutations present in $j$ individuals at time $t$ is then given by
$$
S_j(t) :=\sum_{i\in\mathcal{M}_t}1_{\{C^i(t)=j\}}.
$$
The vector $(S_j(t))_{j \geq 1}$ is the site frequency spectrum (SFS) of neutral mutations at time $t$. 
We also define the total number of mutations present at time $t$ as
\begin{align*} 
M(t) := \sum_{j=1}^\infty S_j(t).
\end{align*}
The goal of this work is to establish first-order limit theorems for $S_j(t)$ and $M(t)$, evaluated either at the fixed time $t$ as $t \to \infty$ or at the random time $\tau_N$ as $N \to \infty$.

\section{Results}

\subsection{General case}

Our main result, Theorem \ref{thm:mainresult}, provides large-time and large-size 
first-order asymptotics for the SFS of neutral mutations conditional on nonextinction.
A proof sketch is given in Section \ref{sec:proofofmainresult} and the proof details are carried out in Sections \ref{app:lemma1}--\ref{app:lemma4}.

\begin{theorem} \label{thm:mainresult}
\begin{enumerate}[(1)]
    \item Conditional on $\Omega_\infty$, 
    \begin{align} \label{eq:fixedtimemainresult}
        \lim_{t \to \infty} e^{-\lambda t} S_j(t) = \nu Y \int_0^\infty e^{-\lambda s} p_j(s) ds, \quad j \geq 1,
    \end{align}
    almost surely. Equivalently, with $r_N := (1/\lambda)\log(q N)$, $X := qY$ and $\E[X|\Omega_\infty]=1$,
    \begin{align} \label{eq:fixedtimemainresultalternative}
        \lim_{N \to \infty} N^{-1} S_j(r_N) = \nu  X \int_0^\infty e^{-\lambda s} p_j(s) ds, \quad j \geq 1,
    \end{align}
    almost surely.
    \item  
    Conditional on $\Omega_\infty$,
    \begin{align} \label{eq:fixedsizemainresult}
        \lim_{N \to \infty} N^{-1} S_j(\tau_N) = \nu \int_0^\infty e^{-\lambda s} p_j(s) ds, \quad j \geq 1,
    \end{align}
    almost surely.
\end{enumerate}
\end{theorem}

\begin{proof}
    Section \ref{sec:proofofmainresult} and Sections \ref{app:lemma1}--\ref{app:lemma4}.
\end{proof}

The main difference between the fixed-time result \eqref{eq:fixedtimemainresult} 
and the fixed-size result \eqref{eq:fixedsizemainresult} is that the limit in \eqref{eq:fixedtimemainresult} is a random variable while it is constant in \eqref{eq:fixedsizemainresult}. 
The reason is that the population size at a large, fixed time $t$ 
is dependent on the limiting random variable $Y$ in $ e^{-\lambda t} Z(t) \to Y$,
while the population size at time $\tau_N$ is always approximately $N$.
In expression \eqref{eq:fixedtimemainresultalternative},
the fixed-time result is viewed at the time $r_N$
defined so that
\begin{align*}
    \lim_{N \to \infty} N^{-1} E[Z(r_N)|\Omega_\infty] = 1.
\end{align*}
The point is to show that 
when the result in \eqref{eq:fixedtimemainresult} is viewed at a fixed time
comparable to $\tau_N$,
the mean of the limiting random variable 
becomes equal to the 
fixed-size limit in \eqref{eq:fixedsizemainresult}.

To establish
the fixed-size result \eqref{eq:fixedsizemainresult}, we prove a secondary approximation result for the hitting time $\tau_N$ defined in \eqref{eq:tau_N_def}.
The result, stated as Proposition \ref{thm:passageApprox}, shows that conditional on $\Omega_\infty$, $\tau_N$ is equal to the approximation $t_N$ defined in \eqref{eq:tNdef} up to an $O(1)$ error.

\begin{proposition} \label{thm:passageApprox}
Conditional on $\Omega_\infty$,
    \begin{align} \label{eq:passageApprox}
        \lim_{N\to\infty} |\tau_N-t_N|=0
    \end{align}
    almost surely.
\end{proposition}

\begin{proof}
    Section \ref{app:approxtime}.
\end{proof}

The proof of the fixed-size result \eqref{eq:fixedsizemainresult} combines the fixed-time result \eqref{eq:fixedtimemainresult} with Proposition \ref{thm:passageApprox} as discussed in Section \ref{sec:fixedsizeproof}.
Finally, a simpler version of the argument used to prove Theorem \ref{thm:mainresult} can be used to prove analogous limit theorems for the total number of mutations at time $t$, $M(t)$.

\begin{proposition}\label{thm:totalnummutresult}
\begin{enumerate}[(1)]
    \item Conditional on $\Omega_\infty$, 
  \begin{align} 
    \label{eq:totalmutfixedtime}
        \lim_{t \to \infty} e^{-\lambda t} M(t) = \nu Y \int_0^\infty e^{-\lambda s} (1-p_0(s)) ds 
    \end{align}
almost surely.
    \item 
    Conditional on $\Omega_\infty$,
\begin{align} \label{eq:totalmutfixedsize}
        & \lim_{N \to \infty} N^{-1} M(\tau_N) = \nu \int_0^\infty e^{-\lambda s} (1-p_0(s)) ds
    \end{align}
    almost surely.
\end{enumerate}
\end{proposition}

\begin{proof}
    Section \ref{app:totalmutgeneral}.
\end{proof}

By combining the results of Theorem \ref{thm:mainresult} and Proposition \ref{thm:totalnummutresult}, we obtain the following limits for the proportion of mutations found in $j \geq 1$ 
individuals:
    \begin{align} \label{eq:proportionofmutfixedtime} 
        & \lim_{t \to \infty} \frac{S_j(t)}{M(t)} = \lim_{N \to \infty} \frac{S_j(\tau_N)}{M(\tau_N)}= \frac{\int_0^\infty e^{-\lambda s} p_j(s) ds}{\int_0^\infty e^{-\lambda s} (1-p_0(s)) ds}, \quad j \geq 1.
    \end{align}
In the application Section \ref{sec:application}, we will also be interested in the proportion of mutations found in $j \geq 1$ individuals out of all mutations found in $\geq j$ individuals.
If we define 
\[
M_j(t) := \sum_{k \geq j} S_j(t), \quad j \geq 1, t \geq 0,
\]
as the total number of mutations found in $ \geq j$ individuals,
this proportion is given by
\begin{align} \label{eq:proportionofmutfixedtimemorethanjind}
    \lim_{t \to \infty} \frac{S_j(t)}{M_j(t)} = \lim_{N \to \infty} \frac{S_j(\tau_N)}{M_j(\tau_N)}= \frac{\int_0^\infty e^{-\lambda s} p_j(s) ds}{\int_0^\infty e^{-\lambda s} \big(\sum_{k=j}^\infty p_k(s)\big) ds}, \quad j \geq 1,
\end{align}
since limit theorems for $M_j(t)$ follow from Theorem \ref{thm:mainresult} and Proposition \ref{thm:totalnummutresult} by writing $M_j(t) = M(t) - \sum_{k=1}^{j-1} S_k(t)$.
Note that for both proportions, the fixed-time and fixed-size limits are the same, as the variability in population size at a fixed time has been removed.
Also note that both proportions are independent of the mutation rate $\nu$.
In Section \ref{sec:application}, we show that for the birth-death process, these properties enable us to define a consistent estimator for the extinction probability $p$ which applies both to the fixed-time and fixed-size SFS.

\subsection{Special case: Birth-death process} \label{sec:mainresultbd}

For the special case of the birth-death process, 
we are able to derive explicit expressions for the limits in Theorem \ref{thm:mainresult} and Proposition \ref{thm:totalnummutresult},
as we demonstrate in the following corollary.

\begin{corollary} \label{corollary}
For the birth-death process, conditional on $\Omega_\infty$,
\begin{enumerate}[(1)]
    \item 
    the random variable $Y$ in Theorem \ref{thm:mainresult} 
    has the exponential distribution with mean $1/q$,
    and the fixed-time result \eqref{eq:fixedtimemainresult} can  be written explicitly as
    \begin{align} \label{eq:fixedtimemainresultbd}
    \begin{split}
                \lim_{t \to \infty} e^{-\lambda t} S_j(t) &= \frac{\nu q Y}{\lambda} \int_0^1 (1-py)^{-1} (1-y) y^{j-1} dy \\
        &= \frac{\nu q Y}{\lambda} \sum_{k=0}^\infty \frac{p^k}{(j+k)(j+k+1)}, \quad j\geq 1.
    \end{split}
    \end{align}
    For the special case $p=0$ of a pure-birth or Yule process,
    \begin{align*}
        \lim_{t \to \infty} e^{-\lambda t} S_j(t) 
        = \frac{\nu Y}{\lambda} \frac1{j(j+1)}.
    \end{align*}
    \item 
    the fixed-size result \eqref{eq:fixedsizemainresult} can be written explicitly as
    \begin{align} \label{eq:fixedsizemainresultbd}
    \begin{split}
                \lim_{N \to \infty} N^{-1} S_j(\tau_N) &= \frac{\nu q}{\lambda} \int_0^1 (1-py)^{-1} (1-y) y^{j-1} dy \\
        &= \frac{\nu q}{\lambda} \sum_{k=0}^\infty \frac{p^k}{(j+k)(j+k+1)}, \quad j\geq 1.
    \end{split}
    \end{align} 
    For the pure-birth or Yule process,
     \begin{align} \label{eq:fixedsizemainresultbdYule}
        \lim_{N \to \infty} N^{-1} S_j(\tau_N) = \frac{\nu}{\lambda} \frac1{j(j+1)}.
    \end{align}
    \item the fixed-time result \eqref{eq:totalmutfixedtime} can be written explicitly as
        \begin{align} 
    \label{eq:totalmutfixedtimebd}
        \lim_{t \to \infty} e^{-\lambda t} M(t) = \begin{cases} \dfrac{\nu Y}{\lambda}, & p=0, \\ - \dfrac{\nu q \log(q) Y}{\lambda p}, &0<p<1. \end{cases} 
    \end{align}
    \item the fixed-size result \eqref{eq:totalmutfixedsize} can be written explicitly as
        \begin{align} \label{eq:totalmutfixedsizebd}
        & \lim_{N \to \infty} N^{-1} M(\tau_N) = \begin{cases} \dfrac{\nu}{\lambda}, & p=0, \\ - \dfrac{\nu q \log(q)}{\lambda p}, & 0<p<1. \end{cases}
    \end{align}
\end{enumerate}
\end{corollary}

\begin{proof}
    Section \ref{app:bdsimplification}.
\end{proof}

Similarly, the proportion of mutations found in $j \geq 1$ individuals,
appearing in expression \eqref{eq:proportionofmutfixedtime},
can be written explicitly as 
\begin{align} \label{eq:proportionofmutbd}
     \frac{\int_0^\infty e^{-\lambda s} p_j(s) ds}{\int_0^\infty e^{-\lambda s} (1-p_0(s)) ds}
             &= \begin{cases} \dfrac1{j(j+1)}, & p=0, \\ \displaystyle -\frac{p}{\log(q)} \int_0^{1} (1-p y)^{-1} (1-y) y^{j-1}  dy, & 0<p<1, \end{cases} 
\end{align}
and the proportion of mutations in $j$ individuals out of all mutations in $\geq j$ individuals,
appearing in expression \eqref{eq:proportionofmutfixedtimemorethanjind},
can be written as
\begin{align} \label{eq:proportionofmutbdmorethanj}
    \varphi_j(p) := \frac{\int_0^\infty e^{-\lambda s} p_j(s) ds}{\int_0^\infty e^{-\lambda s} \big(\sum_{k=j}^\infty p_k(s)\big) ds} = \begin{cases} \dfrac{1}{j+1}, & p=0, \\ 1-\dfrac{ \int_0^{1} (1-p y)^{-1} y^{j}  dy}{\int_0^1 (1-py)^{-1} y^{j-1} dy}, & 0<p<1,
    \end{cases}
\end{align}
see Section \ref{app:propfoundinmroethanjcells}.
Note that expressions \eqref{eq:proportionofmutbd} and \eqref{eq:proportionofmutbdmorethanj} give the same proportion for $j=1$.
It can be shown that for any $j \geq 1$, 
$\varphi_j(p)$ is strictly decreasing in $p$ (Section \ref{app:varphidecreasing}).
In Section \ref{sec:application}, we use this fact to 
develop an estimator for the extinction probability $p$.

We showed
in expression (C.1) of \cite{gunnarsson2021exact} that
for $p=0$,
\begin{align*}
    E[S_j(\tau_N)] = \frac{\nu N}{\lambda} \cdot \frac1{j(j+1)}, \quad j=2,\ldots,N-1.
\end{align*}
In other words, 
the fixed-size result  \eqref{eq:fixedsizemainresultbdYule} holds in the mean
even for finite values of $N$, excluding boundary effects at $j=1$ and $j=N$.

\subsection{Connection to previous results}

A nonasymptotic version of 
expression \eqref{eq:fixedtimemainresultbd} (which implies the asymptotic version) has previously been established for the expected value of the fixed-time SFS, first for a semideterministic model in \cite{ohtsuki2017forward} and then for a fully stochastic model in \cite{gunnarsson2021exact}.
We also argued heuristically in \cite{gunnarsson2021exact} that expression \eqref{eq:fixedsizemainresultbd} holds for the fixed-size spectrum in expectation. 
Similar results for the expected value of the SFS have since appeared for example in \cite{morison2023single} and \cite{bonnet2023site}.
The most similar result to Corollary \ref{corollary} in the literature is given by Theorem 2.3 in Lambert \cite{lambert2009allelic}.
In this work, a ranked sample is drawn from a coalescent point process (CPP),
and an almost sure law of large numbers is established as the sample size is sent to infinity.
When Lambert's result is specialized to the case of exponentially distributed lifetimes,
it has the same form as expression \eqref{eq:fixedsizemainresultbd}.
We note that our result \eqref{eq:fixedsizemainresultbd} deals with the SFS of a population which evolves according to a branching process and is stopped at the first time it reaches size $N$.
In addition, we give the fixed-time result \eqref{eq:fixedtimemainresultbd} for a continuous-time birth-death process, with a random limit.
Finally, we establish both fixed-time and fixed-size results for a general offspring distribution in our main result (Theorem \ref{thm:mainresult}),
a case not addressed by Lambert's CPP analysis.

\section{Proof of Theorem \ref{thm:mainresult}} \label{sec:proofofmainresult}

In this section, we sketch the proof of the main result, Theorem \ref{thm:mainresult}.
Proving the fixed-time result \eqref{eq:fixedtimemainresult} represents most of the work,
which is discussed in Sections \ref{sec:firstdecomp} to \ref{sec:finalstepsfixedtime}.
The main idea is to write the site-frequency spectrum process $(S_j(t))_{t \geq 0}$
as the difference of two increasing processes in time,
and to prove limit theorems for the increasing processes.
The fixed-size result \eqref{eq:fixedsizemainresult} follows easily from fixed-time result \eqref{eq:fixedtimemainresult} and 
Proposition \ref{thm:passageApprox}
via the continuous mapping theorem, as is discussed in Section \ref{sec:fixedsizeproof}.

\subsection{Decomposition into increasing processes $S_{j,+}(t)$ and $S_{j,-}(t)$} \label{sec:firstdecomp}

Fix $j \geq 1$. 
The key idea of the proof 
of the fixed-time result \eqref{eq:fixedtimemainresult}
is to decompose the process $(S_j(t))_{t \geq 0}$ into the difference of two increasing processes $(S_{j,+}(t))_{t \geq 0}$ and $(S_{j,-}(t))_{t \geq 0}$.
To describe these processes, 
we first need to establish some notation.

Recall that 
for mutation $i \in {\cal M}_t$ and $s \leq t$,
$C^i(s)$ is the size of the clone containing mutation $i$ at time $s$,
meaning the number of individuals carrying mutation $i$ at time $s$.
Set
$\tau_{j,-}^i(0) := 0$ and define recursively
for $k \geq 1$,
\begin{align*}
    & \tau_{j,+}^i(k) := \inf\{s>\tau_{j,-}^i(k-1) : C^i(s) = j\}, \\
    & \tau_{j,-}^i(k) := \inf\{s>\tau_{j,+}^i(k) : C^i(s) \neq j\}.
\end{align*}
Note that $\tau_{j,+}^i(k)$ is the $k$-th time at which the clone containing mutation $i$ reaches or \inquotes{enters} size $j$, and $\tau_{j,-}^i(k)$ is the $k$-th time at which it leaves or \inquotes{exits} size $j$.
Next, define
\begin{align} \label{eq:It+-ijdef}
I_{j,+}^i(t) := \sum_{\ell=1}^\infty 1_{\{\tau_{j,+}^i(\ell) \leq t\}}, \quad I_{j,-}^i(t) := \sum_{\ell=1}^\infty 1_{\{\tau_{j,-}^i(\ell) \leq t\}},    
\end{align}
as the number of times the clone containing mutation $i$ 
enters and exits size $j$, respectively, up until time $t$.
Then, for each $k \geq 1$, define the increasing processes $(S_{j,+}^k(t))_{t \geq 0}$ and $(S_{j,-}^k(t))_{t \geq 0}$ by
\begin{align} \label{eq:Sj+-kdef}
    & S_{j,+}^k(t) := \sum_{i \in {\cal M}_t} 1_{\{I_{j,+}^i(t)\geq k\}}, \quad S_{j,-}^k(t) := \sum_{i \in {\cal M}_t} 1_{\{I_{j,-}^i(t) \geq k\}}.
\end{align}
These processes keep track of the number of mutations in ${\cal M}_t$ whose clones enter and exit size $j$, respectively, at least $k$ times up until time $t$.
We can now finally define the increasing processes $(S_{j,+}(t))_{t \geq 0}$ and $(S_{j,-}(t))_{t \geq 0}$ as
\begin{align*}
    S_{j,+}(t) := \sum_{k=1}^\infty S_{j,+}^k(t), \quad S_{j,-}(t) := \sum_{k=1}^\infty S_{j,-}^k(t).
\end{align*}
A key observation is that these processes count the total number of instances that a mutation enters and exits size $j$, respectively, up until time $t$.
To see why, note that
\begin{align*}
    \sum_{k=1}^\infty S_{j,+}^k(t) &= \sum_{i \in {\cal M}_t} \sum_{k=1}^\infty   1_{\{I_{j,+}^i(t)\geq k\}} = \sum_{i \in {\cal M}_t} \sum_{k=1}^\infty \sum_{\ell=k}^\infty 1_{\{I_{j,+}^i(t)=\ell\}} \\
    &= \sum_{i \in {\cal M}_t} \sum_{\ell=1}^\infty \sum_{k=1}^\ell 1_{\{I_{j,+}^i(t)=\ell\}} = \sum_{i \in {\cal M}_t} \sum_{\ell=1}^\infty \ell 1_{\{I_{j,+}^i(t)=\ell\}} \\
    &= \sum_{i \in {\cal M}_t} I_{j,+}^i(t).
\end{align*}
Similar calculations hold for
$\sum_{k=1}^\infty S_{j,-}^k(t)$.
Note that $I_{j,+}^i(t) - I_{j,-}^i(t) = 1$ if and only if $C^i(t)=j$, and $I_{j,+}^i(t) - I_{j,-}^i(t)=0$ otherwise.
It follows that
\begin{align} \label{eq:Sjtdecomp}
    S_j(t) = S_{j,+}(t) - S_{j,-}(t).
\end{align} 
The fixed-time result \eqref{eq:fixedtimemainresult} will follow from limit theorems for $S_{j,+}(t)$ and $S_{j,-}(t)$, which in turn follow from approximation results for the subprocesses $S_{j,+}^k(t)$ and $S_{j,-}^k(t)$ for $k \geq 1$.

\subsection{Approximation results for $S_{j,+}^k(t)$ and $S_{j,-}^k(t)$} \label{sec:Sj+klimits}

We begin by establishing approximation results for $S_{j,+}^k(t)$ and $S_{j,-}^k(t)$ for each $k \geq 1$.
First, for the branching process $(Z(t))_{t \geq 0}$ with $Z(0)=1$, 
set $\tau_{j}^-(0) := 0$ and define recursively 
\begin{align} \label{eq:taujbpdef}
\begin{split}
        & \tau_{j}^+(k) := \inf\{s>\tau_{j}^-(k-1): Z(s) = j\}, \\
    & \tau_{j}^-(k) := \inf\{s>\tau_{j}^+(k): Z(s) \neq j\}, \quad k \geq 1.
\end{split}
\end{align}
Set 
\begin{align} \label{eq:pjkplusminusdef}
\begin{split}
       & p_{j,+}^k(t) := P(\tau_j^+(k) \leq t), \quad p_{j,-}^k(t) := P(\tau_j^-(k) \leq t),
\end{split}
\end{align}
which are the probabilities that the branching process enters and exits size $j$, respectively, at least $k$ times up until time $t$.
A key observation is that
\begin{align} \label{eq:p_jdecomp}
    p_j(t)=P(Z(t)=j)=\sum_{k=1}^\infty\big(p_{j,+}^k(t)-p_{j,-}^k(t)\big),
\end{align}
which follows from the fact that
\begin{align*}
    \{Z(t)=j\} &= \bigcup_{k\geq 1}\{\tau_j^+(k)\leq t, \tau_j^-(k)>t\} \\
    &=\bigcup_{k\geq 1}\{\tau_j^+(k)\leq t\}\backslash\{\tau_j^-(k)\leq t\}.
\end{align*}
In addition, since almost surely, $Z(t) \to 0$ or $Z(t) \to \infty$ as $t \to \infty$, there exist $C>0$ and $0<\theta<1$ so that for each $t \geq 0$,
\begin{align} \label{eq:pkbound}
    p_{j,-}^k(t) \leq p_{j,+}^k(t) \leq P(\tau_j^+(k) < \infty) \leq C \theta^k.
\end{align}
Since the discrete-time process embedded in $(Z(t))_{t \geq 0}$ is a random walk, where the transition $j \mapsto j+k-1$ for $j \geq 1$ and $k \geq 0$ is made with probability $u_k$, $\theta$ can be selected independently of $j$.
Indeed, when the population leaves size $j$, it increases in size with probability $\rho := (1-u_0-u_1)/(1-u_1) >0$ since $u_0+u_1<1$ by supercriticality. 
If the population increases in size, the probability it returns to size $j$ is some number $\gamma<1$, since $Z(t) \to \infty$ as $t \to \infty$ with positive probability. 
Both $\rho$ and $\gamma$ are independent of $j$.
Therefore, when the population leaves size $j$, the probability it returns to size $j$ is upper bounded by $\rho\gamma+(1-\rho) = 1 - (1-\gamma)\rho<1$,
a number independent of $j$.
For $u_0>0$ and $j > 1$, we can take $\theta$ as the larger of this value and the probability that the population ever reaches size $j$ starting from one individual, which is upper bounded by $1-u_0<1$.
The constant $C>0$ takes care of the cases $u_0=0$ and $j=1$.

The approximation results for $S_{j,+}^k(t)$ and $S_{j,-}^k(t)$ can be established using almost identical arguments, 
so
it suffices to analyze $S_{j,+}^k(t)$.
Recall that $S_{j,+}^k(t)$ is the number of mutations whose clones enter size $j$ at least $k$ times up until time $t$.
At any time $s \leq t$,
a mutation occurs at rate $\nu Z(s)$,
and with probability $p_{j,+}^k(t-s)$,
its clone enters size $j$ at least $k$ times up until time $t$.
This suggests the approximation
\begin{align} \label{eq:approx1}
    S_{j,+}^k(t) \approx  \nu \int_0^t Z(s) p_{j,+}^k(t-s)ds =: \bar{S}_{j,+}^k(t).
\end{align}
Since $e^{-\lambda t}Z(t) \to Y$ as $t \to \infty$, 
we can further approximate for large $t$,
\begin{align} \label{eq:approx2}
    \bar{S}_{j,+}^k(t) \approx \nu \int_0^t Y e^{\lambda s} p_{j,+}^k(t-s)ds =: \hat{S}_{j,+}^k(t).
\end{align}
For the remainder of the section,
our goal is to 
establish bounds on the $L^1$-error associated with the approximations $S_{j,+}^k(t) \approx  \bar{S}_{j,+}^k(t) \approx \hat{S}_{j,+}^k(t)$.

We first consider the approximation \eqref{eq:approx1}.
For $\Delta>0$, define the Riemann sum
\begin{align} \label{eq:approx3}
    \bar{S}_{j,+,\Delta}^k(t) := \nu \Delta \sum_{\ell=0}^{\lfloor t/\Delta\rfloor} Z(\ell\Delta) p_{j,+}^k(t-\ell\Delta).
\end{align}
Clearly, $\lim_{\Delta\to 0}\bar{S}^k_{j,+,\Delta}(t)=\bar{S}^k_{j,+}(t)$ almost surely. 
In addition, for some $C>0$,
$$
\bar{S}^k_{j,+,\Delta}(t)\leq C t \max_{s \leq t} Z(s).
$$
Since $(Z(s))_{s \geq 0}$ is a nonnegative submartingale, 
we can use Doob's inequality to show that $C t E\big[\max_{s \leq t} Z(s)\big]<\infty$ for each $t \geq 0$.
Therefore, by dominated convergence,
\begin{align*}
    \lim_{\Delta\to 0}E\big|\bar{S}^k_{j,+,\Delta}(t)-\bar{S}^k_{j,+}(t)\big|=0, \quad t \geq 0.
\end{align*}
It then follows from the triangle inequality that
\begin{align} \label{eq:L1approx3}
   E\big|S_{j,+}^k(t)-\bar{S}^k_{j,+}(t)\big| \leq 
    \lim_{\Delta\to 0}E\big|S_{j,+}^k(t)-\bar{S}^k_{j,+,\Delta}(t)\big|, \quad t \geq 0.
\end{align}
To bound the $L^1$-error of the approximation \eqref{eq:approx1}, it suffices to bound the right-hand side of \eqref{eq:L1approx3}.
We accomplish this in the following lemma.

\begin{lemma}
\label{lemma:approx1}
Let $t>0$ and $\Delta>0$. There exist constants $C_1>0$ and $C_2>0$ independent of $t$, $\Delta$ and $k$ such that
\begin{align} \label{eq:lemmaapprox1}
    E\left[\left(S^k_{j,+}(t)-\bar{S}^k_{j,+,\Delta}(t)\right)^2\right]\leq C_1\theta^kte^{\lambda t}+C_2\Delta e^{3\lambda t}.
\end{align}
\end{lemma}

\begin{proof}
    Section \ref{app:lemma1}.
\end{proof}

We next turn to the approximation \eqref{eq:approx2}.
By the triangle inequality and the Cauchy-Schwarz inequality, we can write
\begin{align*}
E\big |\bar{S}^k_{j,+}(t)-\hat{S}^k_{j,+}(t)\big |&\leq \nu \int_0^{t}E\big|Ye^{\lambda s}-Z(s)\big| \,  p^k_{j,+}(t-s)ds\\
&\leq
\nu \int_0^{t}\left(E\left[\left(Ye^{\lambda s}-Z(s)\right)^2\right]\right)^{1/2}p^k_{j,+}(t-s)ds.
\end{align*}
By showing that $E\left[\left(Ye^{\lambda s}-Z(s)\right)^2\right] = C e^{\lambda s}$ for some $C>0$ and applying \eqref{eq:pkbound}, we can obtain the following bound on the $L^1$-error of the approximation  \eqref{eq:approx2}.

\begin{lemma} \label{lemma:approx2}
    \begin{align}
\label{eq:lemmaapprox2}
E\big|\bar{S}^k_{j,+}(t)-\hat{S}^k_{j,+}(t)\big|=O(\theta^ke^{\lambda t/2}).
\end{align}
\end{lemma}

\begin{proof}
    Section \ref{app:lemma2}.
\end{proof}

Finally, from \eqref{eq:L1approx3}, \eqref{eq:lemmaapprox1} and \eqref{eq:lemmaapprox2}, it is straightforward to obtain a bound on the $L^1$-error of the approximation $S_{j,+}^k(t) \approx \hat{S}_{j,+}^k(t)$, which we state as Proposition \ref{prop:combinedapprox}.

\begin{proposition}
\label{prop:combinedapprox}
\begin{align}
   E\big|S_{j,+}^k(t)- \hat{S}_{j,+}^k(t) \big| = O(\theta^{k/2} t^{1/2} e^{\lambda t/2}).
\end{align}
\end{proposition}

\subsection{Limit theorems for $S_{j,+}(t)$ and $S_{j,-}(t)$} \label{sec:Sjklimits}

To establish limit theorems for $S_{j,+}(t)$ and $S_{j,-}(t)$, we define the approximations
\begin{align*}
    \hat{S}_{j,+}(t) :=  \sum_{k=1}^{\infty}\hat{S}^k_{j,+}(t), \quad \hat{S}_{j,-}(t)  := \sum_{k=1}^{\infty}\hat{S}^k_{j,-}(t).
\end{align*}
Focusing on the former approximation, we first argue that $\lim_{t \to \infty} e^{-\lambda t} \hat{S}_{j,+}(t)$ exists.
Indeed, consider the following calculations for $k \geq 1$ and $t \geq 0$, where we use \eqref{eq:pkbound}:
\begin{align*}
    e^{-\lambda t }\hat{S}_{j,+}^k(t) &=  \nu e^{-\lambda t } \int_0^t Ye^{\lambda s} p_{j,+}^k(t-s) ds \\
    &= \nu Y \int_0^t e^{-\lambda s} p_{j,+}^k(s) ds \\
    &\leq 
    CY
    \theta^k.
\end{align*}
The second equality shows that $t \mapsto  e^{-\lambda t }\hat{S}_{j,+}^k(t)$ is an increasing function for each $k \geq 1$, and the inequality shows that the function is bounded above by the summable sequence 
$C Y \theta^k$.
Therefore, $t \mapsto e^{-\lambda t} \hat{S}_{j,+}(t)$ is increasing and bounded above, which implies that $\lim_{t \to \infty} e^{-\lambda t} \hat{S}_{j,+}(t)$ exists. The limit is given by
\begin{align} \label{eq:Sj+limitexpression}
    \lim_{t \to \infty} e^{-\lambda t} \hat{S}_{j,+}(t) = \nu Y  \int_0^\infty e^{-\lambda s} \left(\sum_{k=1}^\infty p_{j,+}^k(s)\right) ds.
\end{align}
We next note that by the triangle inequality and Proposition \ref{prop:combinedapprox},
\begin{align*}
    E\big|S_{j,+}(t)-\hat{S}_{j,+}(t)\big| &\leq \sum_{k=1}^\infty E\big|S^k_{j,+}(t)-\hat{S}^k_{j,+}(t)\big| = O\big(t^{1/2}e^{\lambda t/2}\big),
\end{align*}
which implies that
\begin{align} \label{eq:finiteintegral}
 \int_0^{\infty} e^{-\lambda t} E\big|S_{j,+}(t)-\hat{S}_{j,+}(t)\big|dt < \infty.
\end{align}
Combining \eqref{eq:finiteintegral} with the fact that $(S_{j,+}(t))_{t \geq 0}$ and $(S_{j,-}(t))_{t \geq 0}$ are increasing processes, 
we can 
establish 
almost sure convergence results 
for $e^{-\lambda t}S_{j,+}(t)$ and $e^{-\lambda t}S_{j,-}(t)$.
In the proof, we adapt an argument of Harris (Theorem 21.1 of \cite{harris1964theory}), with the $L^1$ condition \eqref{eq:finiteintegral} replacing an analogous $L^2$ condition used by Harris.

\begin{proposition}
\label{prop:Sj+(t)convergence}
Conditional on $\Omega_\infty$,
\begin{align*}
& \lim_{t\to\infty} e^{-\lambda t} S_{j,+}(t) 
= \nu Y  \int_0^\infty e^{-\lambda s} \left(\sum_{k=1}^\infty p_{j,+}^k(s)\right) ds, \\
& \lim_{t\to\infty} e^{-\lambda t} S_{j,-}(t) = 
\nu Y  \int_0^\infty e^{-\lambda s} \left(\sum_{k=1}^\infty p_{j,-}^k(s)\right) ds,
\end{align*}
almost surely.
\end{proposition}

\begin{proof}
 Section \ref{app:prop2}.
\end{proof}

\subsection{Proof of the fixed-time result \eqref{eq:fixedtimemainresult}}
 \label{sec:finalstepsfixedtime}

To finish the proof of 
the fixed-time result \eqref{eq:fixedtimemainresult},
it suffices to note that by \eqref{eq:p_jdecomp} and
Proposition \ref{prop:Sj+(t)convergence},
\begin{align*}
        \lim_{t \to \infty} e^{-\lambda t} \big({S}_{j,+}(t)- {S}_{j,-}(t)\big) 
    = \nu Y \int_0^\infty e^{-\lambda s} p_j(s) ds.
\end{align*}
Since $S_j(t) = S_{j,+}(t)-S_{j,-}(t)$ by \eqref{eq:Sjtdecomp}, the result follows.

\subsection{Proof of the fixed-size result \eqref{eq:fixedsizemainresult}} \label{sec:fixedsizeproof}

To prove the fixed-size result \eqref{eq:fixedsizemainresult}, we note that by \eqref{eq:fixedtimemainresult}, conditional on $\Omega_\infty$,
\begin{align*}
    \lim_{N \to \infty} e^{-\lambda \tau_N} S_j(\tau_N) = \nu Y \int_0^\infty e^{-\lambda s} p_j(s) ds,
\end{align*}
almost surely.
Since 
$Ne^{-\lambda t_N} = Y$ by \eqref{eq:tNdef},
we also have
\begin{align*}
\lim_{N \to \infty} e^{-\lambda(\tau_N-t_N)} \cdot N^{-1} S_j(\tau_N) &= 
    Y^{-1} \lim_{N \to \infty} e^{-\lambda \tau_N} S_j(\tau_N) \\
    &= \nu \int_0^\infty e^{-\lambda s} p_j(s) ds,
\end{align*}
almost surely.
By Proposition \ref{thm:passageApprox} and the continuous mapping theorem, conditional on $\Omega_\infty$,
\[
\lim_{N \to \infty} e^{-\lambda(\tau_N-t_N)} = 1,
\]
almost surely.
We can therefore conclude that conditional on $\Omega_\infty$,
\begin{align*}
    \lim_{N \to \infty} N^{-1} S_j(\tau_N) =  \nu \int_0^\infty e^{-\lambda s} p_j(s) ds,
\end{align*}
almost surely, which is the desired result.

\section{Application: Estimation of extinction probability and effective mutation rate for birth-death process} \label{sec:application}

We conclude by briefly discussing how for the birth-death process,
our results imply consistent estimators for the extinction probability $p$
and the effective mutation rate $\nu/\lambda$,
given data on the SFS of all mutations found in the population.
While it may in many cases be difficult to sample the entire population, 
these estimators can potentially be applicable for example  to the setting where an entire subclone of cells within a tumor is sampled  \cite{gunnarsson2021exact}
or to an {\em in vitro} setting where a single cell is expanded 2D or 3D culture to a miniaturized version of a tumor.
The estimator for $p$ is based on the long-run proportion of mutations found in one individual.
Recall that by \eqref{eq:proportionofmutfixedtimemorethanjind}, this proportion is the same for the fixed-time and fixed-size SFS.
By setting $j=1$ in \eqref{eq:proportionofmutbd}, the proportion can be written explicitly as (Section \ref{app:estimator})
\begin{align} \label{eq:varphiexpr}
    \varphi_1(p) = \begin{cases} \dfrac12, & p=0, \\ -\dfrac{p+q\log(q)}{p\log(q)}, & 0<p<1, \end{cases}
\end{align}
where we recall that $q=1-p$.
The function $\varphi_1(p)$ is strictly decreasing in $p$ (Section \ref{app:varphidecreasing}) and it takes values in $(0,1/2]$.
If in a given population, 
the proportion of mutations found in one individual is observed to be $x$, we define an estimator for $p$ by applying the inverse function of $\varphi_1$:
\begin{align} \label{eq:estimatordef}
    \widehat{p} = \widehat{p}(x) := \varphi_1^{-1}(x).
\end{align}
Technically, $\varphi_1^{-1}$ is only defined on $(0,1/2]$, whereas the random number $x$ may take any value in $[0,1]$.
This can be addressed by extending the definition of $\varphi_1^{-1}$ so that $\varphi_1^{-1}(x) := \varphi_1^{-1}(1/2) = 0$ for $x>1/2$ and $\varphi_1^{-1}(0) := \lim_{x \to 0^+} \varphi_1^{-1}(x) = 1$.
Since $\varphi_1^{-1}$ so defined is continuous, we can combine \eqref{eq:proportionofmutfixedtime}
and 
\eqref{eq:proportionofmutbd}
with the continuous mapping theorem to see that whether the SFS is observed at a fixed time or a fixed size, the estimator in \eqref{eq:estimatordef} is consistent in the sense that $\widehat{p} \to p$ almost surely as $t \to \infty$ or $N \to \infty$.
In other words, if the population is sufficiently large, its site frequency spectrum can be used to obtain an arbitrarily accurate estimate of $p$.
Then, using the total number of mutations and the current size of the population, an estimate for $\nu/\lambda$ can be derived from \eqref{eq:totalmutfixedtimebd} or \eqref{eq:totalmutfixedsizebd}.
We refer to Section 5 of \cite{gunnarsson2021exact} for a more detailed discussion of this estimator, which includes an application of the estimator to simulated data.

In the preceding discussion, we focused on the proportion of mutations found in one individual for illustration purposes.
The point was to show that it is possible to define consistent estimators for $p$ and $\nu/\lambda$ using the SFS.
If it is difficult to measure the number of mutations found in one individual, one can instead focus on
the proportion of mutations found in $j$ cells out of all mutations found in $\geq j$ cells for some $j>1$, denoted by $\varphi_j(p)$ in \eqref{eq:proportionofmutbdmorethanj}.
As established in Section \ref{app:varphidecreasing}, $\varphi_j(p)$ is strictly decreasing in $p$ for any $j \geq 1$, and it takes values in $(0,1/(j+1)]$.
We can therefore define a consistent estimator for $p$ using the inverse function $\varphi_j^{-1}(p)$.
However, it should be noted that the range of $\varphi_j(p)$ becomes narrower as $j$ increases, which will likely affect the standard deviation of the estimator.

\section{Proofs}

\subsection{Proof of Lemma \ref{lemma:approx1}} \label{app:lemma1}

\begin{proof}
Before considering the quantity of interest $E\big[\left(S^k_{j,+}(t)-\bar{S}^k_{j,+,\Delta}(t)\right)^2\big]$, we perform some preliminary calculations.
Recall that $\mathcal{M}_t$ is the set of mutations generated up until time $t$.
For $\Delta>0$ and any non-negative integer $\ell$ with $\ell\Delta < t$, define $A_{\ell,\Delta}$ to be the set of mutations created in the time interval $\big[\ell\Delta, \min\{(\ell+1)\Delta,t\}\big)$,
and note that 
$$
\mathcal{M}_t=\bigcup_{\ell=0}^{\lfloor t/\Delta\rfloor}A_{\ell,\Delta}.
$$
Define $X_{\ell,\Delta}:=|A_{\ell,\Delta}|$
as the number of mutations created in $\big[\ell\Delta, \min\{(\ell+1)\Delta,t\}\big)$. 
Note that conditional on ${\cal F}_{(\ell+1)\Delta} = \sigma(Z(s); s\leq (\ell+1)\Delta)$,
$$
X_{\ell,\Delta} \sim \mbox{Pois}\left(\nu\int_{\ell\Delta}^{(\ell+1)\Delta}Z(s)ds\right).
$$
Using this fact, we can write
\begin{align} \label{EXell}
    & E[X_{\ell,\Delta}|{\cal F}_{(\ell+1)\Delta}] = \nu\int_{\ell\Delta}^{(\ell+1)\Delta}Z(s)ds = \Delta \nu Z(\ell\Delta) + Y_{\ell,\Delta},
\end{align}
where 
\begin{align*}
    Y_{\ell,\Delta} := \nu\int_{\ell\Delta}^{(\ell+1)\Delta}(Z(s)-Z(\ell\Delta))ds.
\end{align*}
It is straightforward to establish that
\begin{align} \label{eq:EYelldelta}
\begin{split}
& E[Y_{\ell,\Delta}] = E[Z(\ell\Delta)] O(\Delta^2), \\
& E[Y_{\ell,\Delta}Z(\ell\Delta)] = E[Z(\ell\Delta)^2]O(\Delta^2), \\
& E[Y_{\ell,\Delta}^2] = E[Z(\ell\Delta)^2]O(\Delta^3),
\end{split}
\end{align}
from which it follows that
\begin{align} \label{EXell uncond}
E[X_{\ell,\Delta}] = \Delta \nu E[Z(\ell\Delta)] (1+O(\Delta)).
\end{align}
Since
\begin{align} \label{EXellsqcond}
    E[X_{\ell,\Delta}^2|{\cal F}_{(\ell+1)\Delta}] - E[X_{\ell,\Delta}|{\cal F}_{(\ell+1)\Delta}] = E[X_{\ell,\Delta}|{\cal F}_{(\ell+1)\Delta}]^2,
\end{align}
it furthermore follows from \eqref{EXell} and \eqref{eq:EYelldelta} that
\begin{align} \label{EXellsq}
E[X_{\ell,\Delta}^2] - E[X_{\ell,\Delta}]  = \Delta^2 \nu^2 E[Z(\ell\Delta)^2](1+O(\Delta)).     
\end{align}
Recall that for a mutation $i \in {\cal M}_t$, $I_{j,+}^{i}(t)$ is the number of times the clone containing mutation $i$ reaches size $j$ up until time $t$, see \eqref{eq:It+-ijdef}.
Define
$$
W^k_{\ell\Delta,t}(j):=\sum_{i\in A_{\ell,\Delta}}1_{\{I_{j,+}^{i}(t)\geq k\}}
$$
as the number of mutations in $A_{\ell,\Delta}$ whose clones reach size $j$ at least $k$ times up until time $t$.
Note that by the definition of $S_{j,+}^k(t)$ in \eqref{eq:Sj+-kdef},
\begin{align} \label{eq:Sjk+decomp}
    S_{j,+}^k(t) = \sum_{\ell=0}^{\lfloor t/\Delta\rfloor} W^k_{\ell\Delta,t}(j).
\end{align}
For $i\in A_{\ell,\Delta}$, $P(I_{j,+}^{i}(t)\geq k)=p_{j,+}^k(t-\Delta\ell)+O(\Delta)$, where $p_{j,+}^k(t)$ is defined as in \eqref{eq:pjkplusminusdef}.
Here, we note that $p_{j,+}^k(t-s) = p_{j,+}^k(t-\Delta\ell)+O(\Delta)$ for all $s \in [\ell\Delta,(\ell+1)\Delta)$, since there is an $O(\Delta)$ probability of a death event occurring in an interval of length $\Delta$.
Thus, the $O(\Delta)$ term captures the fact that mutations in $A_{\ell,\Delta}$ can occur anywhere within the interval $[\ell\Delta,(\ell+1)\Delta)$, leading to an error in estimating $P(I_{j,+}^{i}(t)\geq k)$ by $p_{j,+}^k(t-\Delta\ell)$.
Therefore, conditional on $X_{\ell,\Delta}$,
$W^k_{\ell\Delta,t}(j)$ is a binomial random variable with parameters $X_{\ell,\Delta}$ and $\pabbrev := p^k_{j,+}(t-\ell\Delta)+O(\Delta)$.
This implies by \eqref{EXell},
\begin{align}
\label{eq:CondMoment}
E[W^k_{\ell\Delta,t}(j)|{\cal F}_{(\ell+1)\Delta}] 
&=E\left[E\left[W^k_{\ell\Delta,t}(j)|X_{\ell,\Delta},{\cal F}_{(\ell+1)\Delta}\right]|{\cal F}_{(\ell+1)\Delta}\right] \nonumber \\
&=\pabbrev E\left[X_{\ell,\Delta}|{\cal F}_{(\ell+1)\Delta}\right]
\nonumber
\\
& =
\Delta \nu \pabbrev Z(\ell\Delta) + \pabbrev Y_{\ell,\Delta}, 
\end{align}
and by \eqref{EXellsq} and \eqref{EXell uncond},
\begin{align}
\label{eq:SecondMoment}
E\left[W^k_{\ell\Delta,t}(j)^2\right] &= \pabbrev^2E\left[X_{\ell,\Delta}^2\right]+\pabbrev\left(1-\pabbrev\right)E\left[X_{\ell,\Delta}\right]\nonumber\\
&=\pabbrev^2(E\left[X_{\ell,\Delta}^2\right]-E\left[X_{\ell,\Delta}\right])+\pabbrev E\left[X_{\ell,\Delta}\right] \nonumber\\
&=
\big(\pabbrev^2\Delta^2\nu^2E\left[Z(\ell\Delta)^2\right]+\pabbrev \Delta \nu E\left[ Z(\ell\Delta)\right]\big)(1+O(\Delta)) \nonumber \\
&= p^k_{j,+}(t-\ell\Delta)^2 \Delta^2\nu^2E\left[Z(\ell\Delta)^2\right] + p^k_{j,+}(t-\ell\Delta) \Delta\nu E\left[Z(\ell\Delta)\right] \nonumber  \\
&\quad + O(\Delta^3) E\left[Z(\ell\Delta)^2\right] + O(\Delta^2) E\left[ Z(\ell\Delta)\right].
\end{align}

We are now ready to begin the main calculations.
First note that by \eqref{eq:approx3} and \eqref{eq:Sjk+decomp},
\begin{align}
\label{eq:SecondMomentDelta}
&E\left[\left(S^k_{j,+}(t)-\bar{S}^k_{j,+,\Delta}(t)\right)^2\right] \nonumber\\
&=
E\left[\left(\sum_{\ell=0}^{\lfloor t/\Delta\rfloor}\left(\nu\Delta Z(\ell\Delta)p^k_{j,+}(t-\ell\Delta)-W^k_{\ell\Delta,t}(j)\right)\right)^2\right] \nonumber\\
&=
\sum_{\ell_2=0}^{\lfloor t/\Delta\rfloor}\sum_{\ell_1=0}^{\lfloor t/\Delta\rfloor}E\left[\left(\nu\Delta Z(\Delta\ell_2)p^k_{j,+}(t-\Delta\ell_2)-W_{\ell_2\Delta,t}^k(j)\right)\right. \nonumber\\
&\qquad\qquad\qquad\;\;\left.\left(\nu\Delta Z(\Delta \ell_1)p^k_{j,+}(t-\Delta \ell_1)-W^k_{\ell_1\Delta,t}(j)\right)\right].
\end{align}
We first consider the diagonal terms in the double sum.
Note that by \eqref{eq:CondMoment} and \eqref{eq:EYelldelta},
\begin{align*}
    E[Z(\ell\Delta) W^k_{\ell\Delta,t}(j)] &= \Delta \nu \pabbrev E[Z(\ell\Delta)^2] + \pabbrev E[Z(\ell\Delta) Y_{\ell,\Delta}] \\
    &= \Delta \nu \pabbrev E[Z(\ell\Delta)^2] (1+O(\Delta)) \\
    &= \Delta \nu p^k_{j,+}(t-\Delta\ell) E[Z(\ell\Delta)^2] + O(\Delta^2) E[Z(\ell\Delta)^2],
\end{align*}
which together with \eqref{eq:SecondMoment} implies that
\begin{align*}
&E\left[\left(\nu\Delta Z(\ell\Delta)p^k_{j,+}(t-\Delta\ell)-W_{\ell\Delta,t}^k(j)\right)^2\right]\\
&=
\nu^2\Delta^2p^k_{j,+}(t-\ell\Delta)^2E[Z(\ell\Delta)^2]-2\nu\Delta p^k_{j,+}(t-\Delta\ell) E[Z(\ell\Delta)W^k_{\ell\Delta,t}(j)]+E[W^k_{\ell\Delta,t}(j)^2] \\
&=
E\left[W^k_{\ell\Delta,t}(j)^2\right]-\nu^2\Delta^2p^k_{j,+}(t-\ell\Delta)^2E[Z(\ell\Delta)^2] + O(\Delta^3) E[Z(\ell\Delta)^2] \\
&=
\nu\Delta p^k_{j,+}(t-\ell\Delta)E[Z(\ell\Delta)] + O(\Delta^3)E[Z(\ell\Delta)^2] + O(\Delta^2) E[Z(\ell\Delta)].
\end{align*}
Since $E[Z(t)] = e^{ \lambda t}$ and $E[Z(t)^2] = O(e^{2\lambda t})$, we can write
\begin{align*}
&E\left[\left(\nu\Delta Z(\ell\Delta)p^k_{j,+}(t-\Delta\ell)-W_{\ell\Delta,t}^k(j)\right)^2\right]\\
&= \nu\Delta p^k_{j,+}(t-\ell\Delta)E[Z(\ell\Delta)] + O\big(\Delta^3 e^{2\lambda\ell\Delta}\big) + O\big(\Delta^2 e^{\lambda\ell\Delta}\big).
\end{align*}
Next, we consider the cross terms for $\ell_1<\ell_2$:
\begin{align*}
&E\left[\left(\nu\Delta Z(\Delta\ell_2)p^k_{j,+}(t-\Delta\ell_2)-W^k_{\ell_2\Delta,t}(j)\right)\left(\nu\Delta Z(\Delta \ell_1)p^k_{j,+}(t-\Delta \ell_1)-W^k_{\ell_1\Delta,t}(j)\right)\right]\\
&=
\nu\Delta p^k_{j,+}(t-\Delta \ell_1)E\left[Z(\Delta \ell_1)\left(\nu\Delta Z(\Delta\ell_2)p^k_{j,+}(t-\Delta\ell_2)-W^k_{\ell_2\Delta,t}(j)\right)\right]\\
&\quad -
E\left[W^k_{\ell_1\Delta,t}(j)\left(\nu\Delta Z(\Delta\ell_2)p^k_{j,+}(t-\Delta\ell_2)-W^k_{\ell_2\Delta,t}(j)\right)\right].
\end{align*}
To analyze the first term in the difference above, note that by \eqref{eq:CondMoment},
\begin{align*}
    & E\left[Z(\Delta \ell_1)\left(\nu\Delta Z(\Delta\ell_2)p^k_{j,+}(t-\Delta\ell_2)-W^k_{\ell_2\Delta,t}(j)\right)\right] \\
    & =E\left[Z(\Delta \ell_1) \big(\nu\Delta Z(\Delta\ell_2)p^k_{j,+}(t-\Delta\ell_2)-E\left[W^k_{\ell_2\Delta,t}(j)|{\cal F}_{(\ell_2+1)\Delta}\right]\big)\right] \\
    &=  -\left(O(\Delta^2) E\left[Z(\Delta \ell_1) Z(\Delta\ell_2)\right] + \pabbrevvv E\left[Z(\Delta\ell_1) Y_{\ell_2,\Delta}\right]\right).
\end{align*}
It is straightforward to show that
\begin{align*}
    E\left[Z(\Delta\ell_1) Y_{\ell_2,\Delta}\right] =  O\left(\Delta^2 e^{\lambda \ell_1\Delta} e^{\lambda \ell_2\Delta}\right),
\end{align*}
and it thus follows from Cauchy-Schwarz that
\begin{align*}
    & E\left[Z(\Delta \ell_1)\left(\nu\Delta Z(\Delta\ell_2)p^k_{j,+}(t-\Delta\ell_2)-W^k_{\ell_2\Delta,t}(j)\right)\right] 
    = O\left(\Delta^2 e^{\lambda \ell_1\Delta} e^{\lambda \ell_2\Delta}\right).
\end{align*}
The cross terms for $\ell_1 < \ell_2$ can therefore be written as
\begin{align*}
&E\left[\left(\nu\Delta Z(\Delta\ell_2)p^k_{j,+}(t-\Delta\ell_2)-W^k_{\ell_2\Delta,t}(j)\right)\left(\nu\Delta Z(\Delta \ell_1)p^k_{j,+}(t-\Delta \ell_1)-W^k_{\ell_1\Delta,t}(j)\right)\right]\\
&=
E\left[W^k_{\ell_1\Delta,t}(j)\left(W^k_{\ell_2\Delta,t}(j)-\nu\Delta Z(\Delta\ell_2)p^k_{j,+}(t-\Delta\ell_2)\right)\right] + O\left(\Delta^3 e^{\lambda \ell_1 \Delta} e^{\lambda \ell_2 \Delta}\right).
\end{align*}
Using the preceding analysis, we can now rewrite \eqref{eq:SecondMomentDelta} as 
\begin{align}
\label{eq:SecondMomentDelta2}
&E\left[\left(S^k_{j,+}(t)-\bar{S}^k_{j,+,\Delta}(t)\right)^2\right] \nonumber \\
&=\nu\Delta\sum_{\ell=0}^{\lfloor t/\Delta\rfloor}p^k_{j,+}(t-\ell\Delta)E[Z(\ell\Delta)]\nonumber\\
&\quad +
2\sum_{\ell_1<\ell_2} E\left[W^k_{\ell_1\Delta,t}(j)\left(W^k_{\ell_2\Delta,t}(j)-\nu\Delta Z(\Delta\ell_2)p^k_{j,+}(t-\Delta\ell_2)\right)\right] + O\left(\Delta e^{2\lambda t}\right),
\end{align}
where we use that
\begin{align*}
    \Delta \sum_{\ell=0}^{\lfloor t/\Delta\rfloor} e^{\lambda \ell \Delta} = \Delta \sum_{\ell=0}^{\lfloor t/\Delta\rfloor} (e^{\lambda \Delta})^{\ell} = \frac{\Delta\left((e^{\lambda\Delta})^{\lfloor t/\Delta\rfloor+1}-1\right)}{e^{\lambda \Delta}-1} 
    \leq \frac{\Delta e^{\lambda t}(1+O(\Delta))}{\lambda \Delta + O(\Delta^2)} 
    \leq Ce^{\lambda t}
\end{align*}
for some constant $C>0$.
The remainder of the proof will focus on bounding the off-diagonal terms
\begin{align} \label{eq:offdiagonalterms}
E\left[W^k_{\ell_1\Delta,t}(j)\left(W^k_{\ell_2\Delta,t}(j)-\nu\Delta Z(\Delta\ell_2)p^k_{j,+}(t-\Delta\ell_2)\right)\right].
\end{align}
We begin with the following lemma, which shows that in the limit as $\Delta\to 0$, we can ignore the possibility of multiple mutations in time intervals of length $\Delta$.

\begin{lemma}
\label{lemma:IgnoreMultipleMuts}
For $\ell_1<\ell_2$, $\Delta>0$ and $t>0$,
\begin{align*}
& E[W^k_{\ell_2\Delta,t}(j)W^k_{\ell_1\Delta,t}(j)] \\
&=
E[W^k_{\ell_2\Delta,t}(j)W^k_{\ell_1\Delta,t}(j);X_{\ell_1,\Delta}=1,X_{\ell_2,\Delta}=1]
+O\left(e^{\lambda\Delta \ell_1}e^{2\lambda \Delta\ell_2}\Delta^3\right)   , \\
& E[Z(\ell_2\Delta)W^k_{\ell_1\Delta,t}(j)]\\
&=E[Z(\ell_2\Delta);X_{\ell_1,\Delta}=1,W^k_{\ell_1\Delta,t}(j)=1]+{O\left(e^{\lambda\Delta \ell_1}e^{2\lambda \Delta\ell_2}\Delta^2\right)}.
\end{align*}
\end{lemma}

\begin{proof}
    Section \ref{app:lemma3}.
\end{proof}

By Lemma \ref{lemma:IgnoreMultipleMuts}, instead of \eqref{eq:offdiagonalterms} we can study the simpler difference
\begin{align}
\label{eq:OffDiag}
&P(X_{\ell_1,\Delta}=1,W^k_{\ell_1\Delta,t}(j)=1,X_{\ell_2,\Delta}=1,W^k_{\ell_2\Delta,t}(j)=1) \nonumber \\
&-\nu\Delta p^k_{j,+}(t-\Delta\ell_2)E[Z(\ell_2\Delta);X_{\ell_1,\Delta}=1,W^k_{\ell_1\Delta,t}(j)=1].
\end{align}
For ease of notation, define
\begin{align*}
    C_{\ell\Delta,t}^k(j) := \{X_{\ell,\Delta}=1,W^k_{\ell\Delta,t}(j)=1\},
\end{align*}
which is the event that exactly one mutation occurs in $[\ell\Delta,(\ell+1)\Delta)$ and that the clone carrying this mutation reaches size $j$ at least $k$ times up until time $t$.
Also define
\begin{align*}
& I_1(\ell_1,\ell_2):=P\left(C_{\ell_1\Delta,t}^k(j),C_{\ell_2\Delta,t}^k(j)\right), \\
& I_2(\ell_1,\ell_2):=\nu\Delta p^k_{j,+}(t-\Delta\ell_2)E[Z(\ell_2\Delta);C_{\ell_1\Delta,t}^k(j)],
\end{align*}
where we note that \eqref{eq:OffDiag} is $I_1(\ell_1,\ell_2)-I_2(\ell_1,\ell_2)$.
First consider the $I_2(\ell_1,\ell_2)$ term,
\begin{align*}
& \frac{I_2(\ell_1,\ell_2)}{\nu\Delta p^k_{j,+}(t-\Delta\ell_2)} \\
&=E[Z(\ell_2\Delta);C_{\ell_1\Delta,t}^k(j)] \\
&=\sum_{m=1}^{\infty}mP\left(Z(\Delta\ell_2)=m,C_{\ell_1\Delta,t}^k(j)\right)\\
&=
\sum_{m=1}^{\infty}\sum_{n=1}^\infty m P\left(Z(\Delta\ell_2)=m,C_{\ell_1\Delta,t}^k(j),Z(\Delta \ell_1)=n\right)\\
&=
\sum_{n=1}^\infty P(C_{\ell_1\Delta,t}^k(j),Z(\Delta \ell_1)=n) \cdot \sum_{m=1}^{\infty} mP\left(Z(\Delta\ell_2)=m| C_{\ell_1\Delta,t}^k(j),Z(\Delta \ell_1)=n\right),
\end{align*}
which implies that
\begin{align*}
I_2(\ell_1,\ell_2) &=  \nu\Delta p^k_{j,+}(t-\Delta\ell_2) \sum_{n=1}^{\infty} P(C_{\ell_1\Delta,t}^k(j),Z(\Delta \ell_1)=n)\\
&\quad\cdot \sum_{m=1}^\infty mP(Z(\Delta\ell_2)=m|C_{\ell_1\Delta,t}^k(j),Z(\Delta \ell_1)=n).
\end{align*}
Next consider the $I_1(\ell_1,\ell_2)$ term,
\begin{align*}
I_1(\ell_1,\ell_2 &)=P(C_{\ell_1\Delta,t}^k(j),C_{\ell_2\Delta,t}^k(j))\\
&=
\sum_{n=1}^\infty P(C_{\ell_2\Delta,t}^k(j)|Z(\Delta \ell_1)=n,C_{\ell_1\Delta,t}^k(j)) P(C_{\ell_1\Delta,t}^k(j),Z(\Delta \ell_1)=n) \\
&=
\sum_{n=1}^\infty P(C_{\ell_1\Delta,t}^k(j),Z(\Delta \ell_1)=n) \\
&\quad\cdot \sum_{m=1}^\infty P(C_{\ell_2\Delta,t}^k(j)|Z(\Delta\ell_2)=m,Z(\Delta \ell_1)=n,C_{\ell_1\Delta,t}^k(j))\\
&\qquad\quad\;\cdot P(Z(\Delta\ell_2)=m|C_{\ell_1\Delta,t}^k(j),Z(\Delta \ell_1)=n).
\end{align*}
We can therefore write
\begin{align} \label{eq:I1I2diffexpr}
& I_1(\ell_1,\ell_2)-I_2(\ell_1,\ell_2) \nonumber \\
&=
\sum_{n=1}^\infty P(C_{\ell_1\Delta,t}^k(j),Z(\Delta \ell_1)=n) \nonumber \\
&\quad\cdot \sum_{m=1}^\infty P(Z(\Delta\ell_2)=m|C_{\ell_1\Delta,t}^k(j),Z(\Delta \ell_1)=n) \nonumber \\
&\qquad\quad\;\cdot
\left(P(C_{\ell_2\Delta,t}^k(j)|Z(\Delta\ell_2)=m,Z(\Delta \ell_1)=n,C_{\ell_1\Delta,t}^k(j))-m\nu\Delta p^k_{j,+}(t-\Delta\ell_2)\right).
\end{align}
We can use \eqref{eq:I1I2diffexpr} to show that there exists a constant $C>0$ so that
\begin{align} \label{eq:I1I2diffexprtoshow}
    I_1(\ell_1,\ell_2)-I_2(\ell_1,\ell_2) \leq 
    C\Delta^2\theta^ke^{\lambda \Delta\ell_2},
\end{align}
where $\theta$ is obtained from \eqref{eq:pkbound}.
The proof is deferred to the following lemma.

\begin{lemma} \label{lemma:I1I2diff}
For $\ell_1<\ell_2$, $\Delta>0$ and $t>0$, \eqref{eq:I1I2diffexprtoshow} holds.
\end{lemma}

\begin{proof}
    Section \ref{app:lemma4}.
\end{proof}

Returning to \eqref{eq:SecondMomentDelta2}, we can finally use Lemmas \ref{lemma:IgnoreMultipleMuts} and \ref{lemma:I1I2diff} to conclude that there exist positive constants $C_1$, $C_2$ and $C_3$ such that
\begin{align*}
E\left[\left(S^k_{j,+}(t)-\bar{S}^k_{j,+,\Delta}(t)\right)^2\right]&=\nu\Delta\sum_{\ell=0}^{\lfloor t/\Delta\rfloor}p^k_{j,+}(t-\ell\Delta)E[Z(\ell\Delta)]\\
&\quad +
2\sum_{\ell_1<\ell_2}\left(I_1(\ell_1,\ell_2)-I_2(\ell_1,\ell_2)\right)+C_3\Delta e^{3\lambda t}\\
&\leq 
C_1 \theta^ke^{\lambda t}+C_2\theta^kte^{\lambda t}+C_3\Delta e^{3\lambda t},
\end{align*}
where
we use \eqref{eq:pkbound}
in the final step.
This concludes the proof.
\end{proof}

\subsection{Proof of Lemma \ref{lemma:approx2}} \label{app:lemma2}

\begin{proof}
Using that $E[Y|{\cal F}_s] = e^{-\lambda s}Z(s)$, see Section \ref{sec:asymptoticbehavior},
we begin by writing
\begin{align*}
&E\left[\left(Ye^{\lambda s}-Z(s)\right)^2\right]\\
&=
E[Z(s)^2]-2e^{\lambda s}E[YZ(s)]+e^{2\lambda s}E[Y^2]\\
&=
e^{2\lambda s}E[Y^2]-E[Z(s)^2].
\end{align*}
From expression (5) of Chapter III.4 of \cite{athreya2004branching}, we know there exist positive constants $c_1$ and $c_2$ such that
\begin{align}
\label{eq:SecondMomentBP}
E[Z(s)^2]=c_1e^{2\lambda s}-c_2e^{\lambda s}.
\end{align}
If we establish that $E[Y^2]=c_1$,
then it will follow that
\begin{align}
\label{eq:secondmomentbound}
E\left[\left(Ye^{\lambda t}-Z(t)\right)^2\right]=c_2e^{\lambda t},
\end{align}
which is what we need to prove Lemma \ref{lemma:approx2}.
To this end, note that Theorem 1 of IV.11 in \cite{athreya2004branching} implies that $E[(Z(t)e^{-\lambda t})^2]\to E[Y^2]$ as $t\to\infty$. And from \eqref{eq:SecondMomentBP}, we know that 
$$
\lim_{t\to\infty}e^{-2\lambda t}E[Z(t)^2]=c_1.
$$
Therefore, $E[Y^2] = c_1$, which concludes the proof.
\end{proof}

\subsection{Proof of Proposition \ref{prop:Sj+(t)convergence}} \label{app:prop2}

\begin{proof}
Since $S_{j,+}(t)$ is increasing in $t$, 
\begin{align*}
    e^{-\lambda(t+\tau)}S_{j,+}(t+\tau)\geq e^{-\lambda \tau}e^{-\lambda t}S_{j,+}(t), \quad t,\tau\geq 0.
\end{align*}
In Section \ref{sec:Sjklimits}, 
it is shown that
$\hat{S} := \lim_{t \to \infty} e^{-\lambda t}\hat{S}_{j,+}(t)$ exists,
and the limit is positive on $\Omega_\infty$ since $Y>0$,
see \eqref{eq:Sj+limitexpression}.
Suppose there is an $\omega \in \Omega_\infty$ such that 
\begin{comment}
    the statement
\[
\lim\limits_{t\to\infty}e^{-\lambda t}\left(S_{j,b}(t)-\hat{S}_{j,b}(t)\right)=0\]
is not true and we first suppose 
\end{comment}
\begin{align} \label{eq:suppose1}
    \limsup\limits_{t\to\infty}e^{-\lambda t}S_{j,+}(t,\omega) >  
\hat{S}(\omega).
\end{align}
For notational convenience, we will drop the $\omega$ in what follows.
If \eqref{eq:suppose1} is true, there is a $\delta>0$ and a sequence of real numbers $t_1<t_2<\ldots$ such that $t_{i+1} - t_i>\delta/\lambda(2+2\delta)$ and  $e^{-\lambda t_i}S_{j,+}(t_i) > \hat{S}(1+\delta)$ for $i = 1,2,\ldots$.
Then
\begin{equation}
      e^{-\lambda(t_i+\tau)}S_{j,+}(t_i+\tau)\geq e^{-\lambda \tau}e^{-\lambda t_i}S_{j,+}(t_i)\geq (1-\lambda \tau)\hat{S}(1+\delta).
\end{equation}
Also, there exists $t_0$ so that for $t>t_0$, $$e^{-\lambda t}\hat{S}_{j,+}(t)< \hat{S}(1+\delta/2).$$
Therefore, for $t_i>t_0$,
\begin{align*}
    & \int_{t_i}^{t_{i+1}} \left| e^{-\lambda t}S_{j,+}(t)-e^{-\lambda t}\hat{S}_{j,+}(t)\right|dt \geq \int_{t_i}^{t_i+\delta/\lambda(2+2\delta)} \left | e^{-\lambda t}S_{j,+}(t)-e^{-\lambda t}\hat{S}_{j,+}(t)\right|dt\\
    & \geq \int_{t_i}^{t_i+\delta/\lambda(2+2\delta)}\left(e^{-\lambda t}S_{j,+}(t)-e^{-\lambda t}\hat{S}_{j,+}(t)\right)dt\\
    & \geq \hat{S}\int_{0}^{\delta/\lambda(2+2\delta)}\left( (1-\lambda\tau)(1+\delta)-(1+\delta/2)\right)d\tau \\
    &= \hat{S} \cdot \frac{\delta^2}{8\lambda(1+\delta)},
\end{align*}
from which it follows that
\begin{equation*}
    \int_{0}^{\infty} \left| e^{-\lambda t}S_{j,+}(t)-e^{-\lambda t}\hat{S}_{j,+}(t)\right|dt = \infty.
\end{equation*}
By
\eqref{eq:finiteintegral}, we see that the inequality \eqref{eq:suppose1} cannot hold on a set of positive probability.

Now suppose that
\begin{align} \label{eq:suppose2}
    \liminf\limits_{t\to\infty}e^{-\lambda t}S_{j,+}(t,\omega) <  
\hat{S}(\omega)
\end{align} 
for some $\omega \in \Omega_\infty$.
Then there is a sequence of real numbers $t_1<t_2<\ldots$ with $t_{i+1}-t_i>\delta/\lambda(2-\delta)$ and a real number $0< \delta<1$ such that $e^{-\lambda t_i}S_{j,+}(t_i)<(1-\delta)\hat{S}$. Therefore, 
\begin{equation}
    e^{-\lambda (t_i-\tau)}S_{j,+}(t_i-\tau)\leq (1-\delta)\hat{S}e^{\lambda \tau}\leq \frac{(1-\delta)\hat{S}}{1-\lambda\tau}, \quad 0\leq \tau < 1/\lambda.
\end{equation}
Also,
there exists $t_0$ so that for $t>t_0$, 
$$e^{-\lambda t}\hat{S}_{j,+}(t)> (1-\delta/2)\hat{S}.$$
Therefore,
\begin{align*}
    &\int_{t_i}^{t_{i+1}} \left|e^{-\lambda t}S_{j,+}(t) - e^{-\lambda t}\hat{S}_{j,+}(t)\right|dt \geq \int_{t_{i+1}-\delta/\lambda(2-\delta)}^{t_{i+1}} \left|e^{-\lambda t}S_{j,+}(t) - e^{-\lambda t}\hat{S}_{j,+}(t)\right|dt \\
    & \geq  \int_{t_{i+1}-\delta/\lambda(2-\delta)}^{t_{i+1}} \left(  e^{-\lambda t}\hat{S}_{j,+}(t)-e^{-\lambda t}S_{j,+}(t)\right)dt \\
    & \geq \hat{S}\int_{0}^{\delta/\lambda(2-\delta)} \left(  (1-\delta/2)-(1-\delta)/(1-\lambda \tau)\right)d\tau \\
    & = \hat{S} \left(\frac{\delta}{2\lambda}+\frac{1-\delta}{\lambda}\log \frac{2-2\delta}{2-\delta}\right),
\end{align*}
where we can verify that $\frac{\delta}{2\lambda}+\frac{1-\delta}{\lambda}\log \frac{2-2\delta}{2-\delta}>0 $ when $\delta<1$.
Hence 
\begin{equation*}
    \int_{0}^{\infty} \left| e^{-\lambda t}S_{j,+}(t)-e^{-\lambda t}\hat{S}_{j,+}(t)\right|dt = \infty,
\end{equation*}
which allows us to conclude that \eqref{eq:suppose2} cannot hold on a set of positive probability.

We can now conclude that  on $\Omega_\infty$, 
\begin{align*}
    \lim_{t \to \infty} e^{-\lambda t}S_{j,+}(t) = \hat{S}
\end{align*}
almost surely, which is the desired result. \qedhere
\end{proof}

\subsection{Proof of Lemma \ref{lemma:IgnoreMultipleMuts}} \label{app:lemma3}

\begin{proof}
We begin with the proof of the first statement.
It suffices to show that
\begin{align*}
&E[W^k_{\ell_2\Delta,t}(j)W^k_{\ell_1\Delta,t}(j); X_{\ell_2,\Delta}>1]+E[W^k_{\ell_2\Delta,t}(j)W^k_{\ell_1\Delta,t}(j); X_{\ell_1,\Delta}>1] \\
&=O\left(e^{\lambda\Delta \ell_1}e^{2\lambda \Delta\ell_2}\Delta^3\right),    
\end{align*}
with $\ell_1<\ell_2$.
We will only show that the first term satisfies the bound, the proof for the second term being largely the same. We first 
note that since $W^k_{\ell_1\Delta,t}(j)\leq X_{\ell_1,\Delta}$, 
\begin{align*}
&E[W^k_{\ell_2\Delta,t}(j)W^k_{\ell_1\Delta,t}(j)1_{\{X_{\ell_2,\Delta}>1\}}] \\
&=E\left[E[W^k_{\ell_2\Delta,t}(j)W^k_{\ell_1\Delta,t}(j)1_{\{X_{\ell_2,\Delta}>1\}}|\mathcal{F}_{\Delta (\ell_1+1)}]\right]\\
&\leq
E\left[E[X_{\ell_1,\Delta}W^k_{\ell_2\Delta,t}(j)1_{\{X_{\ell_2,\Delta}>1\}}|\mathcal{F}_{\Delta (\ell_1+1)}]\right]\\
&=
E\left[E\left[X_{\ell_1,\Delta}\rvert\mathcal{F}_{\Delta (\ell_1+1)}\right]E\left[W^k_{\ell_2\Delta,t}(j)1_{\{X_{\ell_2,\Delta}>1\}}\Big|\mathcal{F}_{\Delta (\ell_1+1)}\right]\right].
\end{align*}
In the final equality, we use that the number of mutations created in the interval $[\Delta \ell_1, \Delta \ell_1 +\Delta)$ is independent of the number of mutations created in $[\Delta\ell_2, \Delta\ell_2+\Delta)$ and their fate, given the population size up until time $\Delta (\ell_1+1)$. Therefore, using \eqref{EXell}, $W^k_{\ell_2\Delta,t}(j)\leq X_{\ell_2,\Delta}$ and \eqref{EXellsqcond},
\begin{align*}
&E[W^k_{\ell_2\Delta,t}(j)W^k_{\ell_1\Delta,t}(j)1_{\{X_{\ell_2,\Delta}>1\}}] \\
&\leq
E\left[\left(\nu\Delta Z(\Delta \ell_1)+Y_{\ell_1,\Delta}\right)E\left[E\left[W^k_{\ell_2\Delta,t}(j)1_{\{X_{\ell_2,\Delta}>1\}}\Big|\mathcal{F}_{\Delta(\ell_2+1)}\right]\Big|\mathcal{F}_{\Delta (\ell_1+1)}\right]\right]\\
&\leq  E\left[\left(\nu\Delta Z(\Delta \ell_1)+Y_{\ell_1,\Delta}\right) E\left[E\left[X_{\ell_2,\Delta}1_{\{X_{\ell_2,\Delta}>1\}}\Big|\mathcal{F}_{\Delta(\ell_2+1)}\right]\Big|\mathcal{F}_{\Delta (\ell_1+1)}\right]\right]\\
&\leq
 E\left[\left(\nu\Delta Z(\Delta \ell_1)+Y_{\ell_1,\Delta}\right) E\left[E\left[X_{\ell_2,\Delta}(X_{\ell_2,\Delta}-1)\Big|\mathcal{F}_{\Delta(\ell_2+1)}\right]\Big|\mathcal{F}_{\Delta (\ell_1+1)}\right]\right]\\
&= 
E\left[\left(\nu\Delta Z(\Delta \ell_1)+Y_{\ell_1,\Delta}\right) E\left[\left(\nu\Delta Z(\Delta \ell_2)+Y_{\ell_2,\Delta}\right)^2|\mathcal{F}_{\Delta (\ell_1+1)}\right]\right].
\end{align*}
We will only show that
\begin{align*}
    \nu^3 \Delta^3 E\left[Z(\Delta \ell_1)E\left[Z(\Delta\ell_2)^2|\mathcal{F}_{\Delta (\ell_1+1)}\right]\right] = O\left(e^{\lambda \Delta \ell_1} e^{2\lambda \Delta \ell_2} \Delta^3\right),
\end{align*}
since the terms involving $Y_{\ell_1,\Delta}$ and $Y_{\ell_2,\Delta}$ can be handled similarly.
To that end, note that for $s \leq t$,
\begin{align*}
&E[Z(t)^2|\mathcal{F}_{s}]=e^{2\lambda(t-s)}Z(s)^2+\Var\left(Z(t-s)\right)Z(s),
\end{align*}
which implies that
\begin{align*}
&\nu^3 \Delta^3 E\left[Z(\Delta \ell_1)E\left[Z(\Delta\ell_2)^2|\mathcal{F}_{\Delta (\ell_1+1)}\right]\right] \\
&\leq
\nu^3\Delta^3e^{2\lambda\Delta(\ell_2-\ell_1-1)}E[Z(\Delta \ell_1)Z(\Delta (\ell_1+1))^2]\\
&\quad+\nu^3\Delta^3\Var\left(Z(\Delta(\ell_2-\ell_1-1))\right)E[Z(\Delta \ell_1)Z(\Delta (\ell_1+1))] \\
&= \nu^3\Delta^3e^{2\lambda\Delta(\ell_2-\ell_1)}E[Z(\Delta \ell_1)^3]\\
&\quad+ \nu^3\Delta^3e^{2\lambda\Delta(\ell_2-\ell_1-1)}\Var(Z(\Delta))E[Z(\Delta \ell_1)^2]\\
&\quad+ \nu^3\Delta^3\Var\left(Z(\Delta(\ell_2-\ell_1-1))\right)e^{\lambda\Delta}E[Z(\Delta \ell_1)^2].
\end{align*}
The desired result now follows from the assumption that the offspring distribution has a finite third moment and thus $E[Z(t)^3]=O\left(e^{3\lambda t}\right)$ by Lemma 5 of \cite{foo2014escape}.

{
For the second statement, the proof is largely the same. 
Since
\begin{align*}
    & E[Z(\ell_2\Delta)W^k_{\ell_1\Delta,t}(j)] \\
    & = E[Z(\ell_2\Delta);W^k_{\ell_1\Delta,t}(j)=1, X_{\ell_1,\Delta}=1]+ E[Z(\ell_2\Delta)W^k_{\ell_1\Delta,t}(j);X_{\ell_1,\Delta}>1],
\end{align*}
we need to show that
\begin{align*}
E[Z(\ell_2\Delta)W^k_{\ell_1\Delta,t}(j);X_{\ell_1,\Delta}>1] = O\left(e^{\lambda\Delta \ell_1}e^{2\lambda \Delta\ell_2}\Delta^2\right).
\end{align*}
To that end, we note that
\begin{align*}
E[Z(\ell_2\Delta)W^k_{\ell_1\Delta,t}(j);X_{\ell_1,\Delta}>1] &= E\left[E\left[Z(\ell_2\Delta)W^k_{\ell_1\Delta,t}(j)1_{\{X_{\ell_1,\Delta}>1\}}\Big|\mathcal{F}_{\Delta (\ell_1+1)}\right]\right]\\
& \leq  E\left[E\left[Z(\ell_2\Delta)X_{\ell_1,\Delta}1_{\{X_{\ell_1,\Delta}>1\}}\Big|\mathcal{F}_{\Delta (\ell_1+1)}\right]\right]\\
& = E\left[E\left[X_{\ell_1,\Delta}1_{\{X_{\ell_1,\Delta}>1\}}\Big|\mathcal{F}_{\Delta (\ell_1+1)}\right]E\left[Z(\ell_2\Delta)\big|\mathcal{F}_{\Delta (\ell_1+1)}\right]\right]\\
& \leq E\left[E\left[X_{\ell_1,\Delta}(X_{\ell_1,\Delta}-1)\big|\mathcal{F}_{\Delta (\ell_1+1)}\right]E\left[Z(\ell_2\Delta)\big|\mathcal{F}_{\Delta (\ell_1+1)}\right]\right]\\
& = E\left[\left(\nu\Delta Z(\Delta \ell_1)+Y_{\ell_1,\Delta}\right)^2Z(\Delta (\ell_1+1))e^{\lambda(\ell_2-\ell_1-1)\Delta}\right].
\end{align*}
The first inequality holds because $W^k_{\ell_1\Delta,t}(j)\leq X_{\ell_1,\Delta}$. The second equality holds because $X_{\ell_1,\Delta}$ and $Z(\ell_2\Delta)$ are independent conditional on $\mathcal{F}_{\Delta (\ell_1+1)}$. The last equality is obtained from \eqref{EXellsqcond} and \eqref{EXell}.  We will only show that
\begin{align*}
\nu^2\Delta^2E\left[Z(\Delta \ell_1)^2Z(\Delta (\ell_1+1)) \right]e^{\lambda(\ell_2-\ell_1-1)\Delta} = O\left(e^{\lambda\Delta \ell_1}e^{2\lambda \Delta\ell_2}\Delta^2\right)
\end{align*}
since the terms involving $Y_{\ell_1,\Delta}$ can be handled similarly.
For this, it suffices to note that
\begin{align*}
    E\left[Z(\Delta \ell_1)^2Z(\Delta (\ell_1+1)) \right]e^{\lambda(\ell_2-\ell_1-1)\Delta} &= E\left[Z(\Delta \ell_1)^3 \right]e^{\lambda(\ell_2-\ell_1)\Delta}\\
    & = O\left(e^{2\lambda\Delta \ell_1}e^{\lambda \Delta\ell_2}\right) \\
    &= O\left(e^{\lambda\Delta \ell_1}e^{2\lambda \Delta\ell_2}\right).
    \qedhere
\end{align*}
}

\end{proof}

\subsection{Proof of Lemma \ref{lemma:I1I2diff}} \label{app:lemma4}

\begin{proof}
Let $\ell$ be a positive integer and let $s>0$ such that $\ell\Delta+\Delta< s$.
On the event $\{X_{\ell,\Delta}=1\}$, define $D^j_{\ell\Delta}(s)$ to be the number of disjoint time intervals in $[0,s]$ that the mutation created in the interval $[\ell\Delta,(\ell+1)\Delta)$ is present in $j$ individuals, and let $B_{\ell\Delta}(s)$ be the number of individuals alive at time $s$
descended from that mutation.
Note that 
\begin{align*}
C_{\ell\Delta,t}^k(j) = \{X_{\ell,\Delta}=1, W^k_{\ell\Delta,t}(j)=1\} = \{X_{\ell,\Delta}=1,D^j_{\ell\Delta}(t)\geq k\}.
\end{align*}
On $\{X_{\ell_1,\Delta}=1,X_{\ell_2,\Delta}=1\}$ with $\ell_1<\ell_2$, let $A$ denote the event 
that the mutation created in $[\ell_2\Delta,(\ell_2+1)\Delta)$ occurs in
the clone started by the mutation created in  $[\ell_1\Delta,(\ell_1+1)\Delta)$.

We now consider the first term inside the parenthesis in \eqref{eq:I1I2diffexpr},  and break it up based on the value of $B_{\ell_1\Delta}(\ell_2\Delta)$ and whether $A$ occurs or not.
\begin{align*}
&P(C_{\ell_2\Delta,t}^k(j)|Z(\Delta\ell_2)=m,Z(\Delta \ell_1)=n,C_{\ell_1\Delta,t}^k(j))\\
&=
\sum_{i=1}^mP(C_{\ell_2\Delta,t}^k(j),B_{\ell_1\Delta}(\ell_2\Delta)=i|Z(\Delta\ell_2)=m,Z(\Delta \ell_1)=n,C_{\ell_1\Delta,t}^k(j))\\
&=
\sum_{i=1}^mP(C_{\ell_2\Delta,t}^k(j),B_{\ell_1\Delta}(\ell_2\Delta)=i,A|Z(\Delta\ell_2)=m,Z(\Delta \ell_1)=n,C_{\ell_1\Delta,t}^k(j))\\
&\quad+
\sum_{i=1}^mP(C_{\ell_2\Delta,t}^k(j),B_{\ell_1\Delta}(\ell_2\Delta)=i,A^c|Z(\Delta\ell_2)=m,Z(\Delta \ell_1)=n,C_{\ell_1\Delta,t}^k(j)).
\end{align*}
Note that
\begin{align*}
   & P(C_{\ell_2\Delta,t}^k(j),B_{\ell_1\Delta}(\ell_2\Delta)=i,A|Z(\Delta\ell_2)=m,Z(\Delta \ell_1)=n,C_{\ell_1\Delta,t}^k(j)) \\
&= P(C_{\ell_2\Delta,t}^k(j),A|Z(\Delta\ell_2)=m,Z(\Delta \ell_1)=n,C_{\ell_1\Delta,t}^k(j),B_{\ell_1\Delta}(\ell_2\Delta)=i) \\
& \quad\; \cdot P(B_{\ell_1\Delta}(\ell_2\Delta)=i|Z(\Delta\ell_2)=m,Z(\Delta \ell_1)=n,C_{\ell_1\Delta,t}^k(j)) \\
&= P(C_{\ell_2\Delta,t}^k(j),A,D^j_{\ell_1\Delta}(t)\geq k|Z(\Delta\ell_2)=m,Z(\Delta \ell_1)=n,X_{\ell_1,\Delta}=1,B_{\ell_1\Delta}(\ell_2\Delta)=i) \\
&\quad\; \cdot \frac{P(B_{\ell_1\Delta}(\ell_2\Delta)=i|Z(\Delta\ell_2)=m,Z(\Delta \ell_1)=n,C_{\ell_1\Delta,t}^k(j))}{P(D^j_{\ell_1\Delta}(t)\geq k|Z(\Delta\ell_2)=m,Z(\Delta \ell_1)=n,X_{\ell_1,\Delta}=1,B_{\ell_1\Delta}(\ell_2\Delta)=i)} \\ 
&\leq P(C_{\ell_2\Delta,t}^k(j),A|Z(\Delta\ell_2)=m,Z(\Delta \ell_1)=n,X_{\ell_1,\Delta}=1,B_{\ell_1\Delta}(\ell_2\Delta)=i) \\
&\quad\; \cdot \frac{P(B_{\ell_1\Delta}(\ell_2\Delta)=i|Z(\Delta\ell_2)=m,Z(\Delta \ell_1)=n,C_{\ell_1\Delta,t}^k(j))}{P(D^j_{\ell_1\Delta}(t)\geq k|Z(\Delta\ell_2)=m,Z(\Delta \ell_1)=n,X_{\ell_1,\Delta}=1,B_{\ell_1\Delta}(\ell_2\Delta)=i)}, \\
\end{align*}
and
\begin{align*}
    & P(C_{\ell_2\Delta,t}^k(j),A|Z(\Delta\ell_2)=m,Z(\Delta \ell_1)=n,X_{\ell_1,\Delta}=1,B_{\ell_1\Delta}(\ell_2\Delta)=i)
    = i \nu \Delta \pabbrevvv. 
\end{align*}
Also note that
\begin{align*}
    & P(C_{\ell_2\Delta,t}^k(j),B_{\ell_1\Delta}(\ell_2\Delta)=i,A^c|Z(\Delta\ell_2)=m,Z(\Delta \ell_1)=n,C_{\ell_1\Delta,t}^k(j)) \\
    &= P(C_{\ell_2\Delta,t}^k(j),A^c|Z(\Delta\ell_2)=m,Z(\Delta \ell_1)=n,C_{\ell_1\Delta,t}^k(j),B_{\ell_1\Delta}(\ell_2\Delta)=i) \\
    &\quad\cdot P(B_{\ell_1\Delta}(\ell_2\Delta)=i|Z(\Delta\ell_2)=m,Z(\Delta \ell_1)=n,C_{\ell_1\Delta,t}^k(j)) \\
    &= (m-i)\nu\Delta \pabbrevvv 
    P(B_{\ell_1\Delta}(\ell_2\Delta)=i|Z(\Delta\ell_2)=m,Z(\Delta \ell_1)=n,C_{\ell_1\Delta,t}^k(j)).
\end{align*}
It follows that
\begin{align*}
&P(C_{\ell_2\Delta,t}^k(j)|Z(\Delta\ell_2)=m,Z(\Delta \ell_1)=n,C_{\ell_1\Delta,t}^k(j)) \\
&\leq
 \nu \Delta  \pabbrevvv  \\
 &\quad\cdot
 \left(\sum_{i=1}^mi\frac{P(B_{\ell_1\Delta}(\ell_2\Delta)=i|Z(\Delta\ell_2)=m,Z(\Delta \ell_1)=n,C_{\ell_1\Delta,t}^k(j))}{P(D^j_{\ell_1\Delta}(t)\geq k|Z(\Delta\ell_2)=m,Z(\Delta \ell_1)=n,X_{\ell_1,\Delta} = 1,B_{\ell_1\Delta}(\ell_2\Delta)=i)} \right. \\
 &\quad\quad\;\;+
\left.
\sum_{i=1}^m(m-i) P(B_{\ell_1\Delta}(\ell_2\Delta)=i|Z(\Delta\ell_2)=m,Z(\Delta \ell_1)=n,C_{\ell_1\Delta,t}^k(j))\right) \\
&\leq
\nu 
\Delta  \pabbrevvv  \\
&\quad\cdot\left(m+
\sum_{i=1}^m i \frac{P(B_{\ell_1\Delta}(\ell_2\Delta)=i|Z(\Delta\ell_2)=m,Z(\Delta \ell_1)=n,C_{\ell_1\Delta,t}^k(j))}{P(D^j_{\ell_1\Delta}(t)\geq k|Z(\Delta\ell_2)=m,Z(\Delta \ell_1)=n,X_{\ell_1,\Delta} = 1,B_{\ell_1\Delta}(\ell_2\Delta)=i)}\right).
\end{align*}
Going back to \eqref{eq:I1I2diffexpr}, we can then derive the upper bound
\begin{align*}
&I_1(\ell_1,\ell_2)-I_2(\ell_1,\ell_2) \\
&\leq \nu\Delta  \pabbrevvv  \sum_{n=1}^\infty P(C_{\ell_1\Delta,t}^k(j),Z(\Delta \ell_1)=n) \\
&\quad\cdot \sum_{m=1}^\infty P(Z(\Delta\ell_2)=m|C_{\ell_1\Delta,t}^k(j),Z(\Delta \ell_1)=n)\\
&\quad\cdot
\sum_{i=1}^m i \frac{P(B_{\ell_1\Delta}(\ell_2\Delta)=i|Z(\Delta\ell_2)=m,Z(\Delta \ell_1)=n,C_{\ell_1\Delta,t}^k(j))}{P(D^j_{\ell_1\Delta}(t)\geq k|Z(\Delta\ell_2)=m,Z(\Delta \ell_1)=n,X_{\ell_1,\Delta} = 1,B_{\ell_1\Delta}(\ell_2\Delta)=i)}.
\end{align*}
Note that 
\begin{align*}
& P(Z(\Delta\ell_2)=m|C_{\ell_1\Delta,t}^k(j),Z(\Delta \ell_1)=n) \\
&\cdot P(B_{\ell_1\Delta}(\ell_2\Delta)=i|Z(\Delta\ell_2)=m,Z(\Delta \ell_1)=n,C_{\ell_1\Delta,t}^k(j))\\
&= \frac{P(B_{\ell_1\Delta}(\ell_2\Delta)=i,Z(\Delta\ell_2)=m,Z(\Delta \ell_1)=n,C_{\ell_1\Delta,t}^k(j))}{ P(C_{\ell_1\Delta,t}^k(j),Z(\Delta \ell_1)=n)}
\end{align*}
and
\begin{align*}
& \frac{P(B_{\ell_1\Delta}(\ell_2\Delta)=i,Z(\Delta\ell_2)=m,Z(\Delta \ell_1)=n,C_{\ell_1\Delta,t}^k(j))}{P(D^j_{\ell_1\Delta}(t)\geq k|Z(\Delta\ell_2)=m,Z(\Delta \ell_1)=n,X_{\ell_1,\Delta} = 1,B_{\ell_1\Delta}(\ell_2\Delta)=i)} \\
&= \frac{P(B_{\ell_1\Delta}(\ell_2\Delta)=i,Z(\Delta\ell_2)=m,Z(\Delta \ell_1)=n,X_{\ell_1,\Delta}=1,D^j_{\ell_1\Delta}(t)\geq k)}{P(D^j_{\ell_1\Delta}(t)\geq k|Z(\Delta\ell_2)=m,Z(\Delta \ell_1)=n,X_{\ell_1,\Delta} = 1,B_{\ell_1\Delta}(\ell_2\Delta)=i)} \\
&= P(Z(\Delta\ell_2)=m,Z(\Delta \ell_1)=n,X_{\ell_1,\Delta} = 1,B_{\ell_1\Delta}(\Delta\ell_2)=i).
\end{align*}
It follows that
\begin{align*}
&  P(Z(\Delta\ell_2)=m|C_{\ell_1\Delta,t}^k(j),Z(\Delta \ell_1)=n) \\
& \cdot \frac{P(B_{\ell_1\Delta}(\ell_2\Delta)=i|Z(\Delta\ell_2)=m,Z(\Delta \ell_1)=n,C_{\ell_1\Delta,t}^k(j))}{P(D^j_{\ell_1\Delta}(t)\geq k|Z(\Delta\ell_2)=m,Z(\Delta \ell_1)=n,X_{\ell_1,\Delta} = 1,B_{\ell_1\Delta}(\ell_2\Delta)=i)} \\
&= \frac{P(Z(\Delta\ell_2)=m,Z(\Delta \ell_1)=n,X_{\ell_1,\Delta} = 1,B_{\ell_1\Delta}(\Delta\ell_2)=i)}{P(C_{\ell_1\Delta,t}^k(j),Z(\Delta \ell_1)=n)},
\end{align*}
which allows us to write
\begin{align*}
&I_1(\ell_1,\ell_2)-I_2(\ell_1,\ell_2) \\
&\leq\Delta\nu  \pabbrevvv  \sum_{n=1}^\infty  \sum_{m=1}^\infty \sum_{i=1}^m i  P(Z(\Delta\ell_2)=m,Z(\Delta \ell_1)=n,X_{\ell_1,\Delta} = 1,B_{\ell_1\Delta}(\Delta\ell_2)=i).
\end{align*}
Now,
\begin{align*}
& P(Z(\Delta\ell_2)=m,Z(\Delta \ell_1)=n,X_{\ell_1,\Delta} = 1,B_{\ell_1\Delta}(\Delta\ell_2)=i) \\
&= P(Z(\Delta\ell_2)=m, B_{\ell_1\Delta}(\Delta\ell_2)=i| Z(\Delta \ell_1)=n,X_{\ell_1,\Delta} = 1) \\
&\quad\cdot P(X_{\ell_1,\Delta}=1|Z(\Delta \ell_1)=n) P(Z(\Delta \ell_1)=n) \\
&= n\nu\Delta P(Z(\Delta \ell_1)=n) p_i(\Delta (\ell_2-\ell_1)) p_{n-1,m-i}(\Delta(\ell_2-\ell_1)),
\end{align*}
where we recall that $p_{n,m}(t) = P(Z(t)=m|Z(0)=n)$ and $p_{m}(t) = p_{1,m}(t)$.
It follows that there exists a constant $C>0$ so that
\begin{align*}
&I_1(\ell_1,\ell_2)-I_2(\ell_1,\ell_2) \\
&\leq \nu^2\Delta^2  \pabbrevvv  \sum_{n=1}^\infty  nP(Z(\Delta \ell_1)=n)  \sum_{m=1}^\infty \sum_{i=1}^m i p_{n-1,m-i}(\Delta(\ell_2-\ell_1)) p_i(\Delta (\ell_2-\ell_1)) \\
&=
\nu^2\Delta^2  \pabbrevvv  \sum_{n=1}^\infty nP(Z(\Delta \ell_1)=n)\sum_{i=1}^\infty  ip_i(\Delta (\ell_2-\ell_1)) \sum_{m=i}^\infty p_{n-1,m-i}(\Delta(\ell_2-\ell_1))\\
&=
\nu^2\Delta^2  \pabbrevvv  \sum_{n=1}^\infty nP(Z(\Delta \ell_1)=n)\sum_{i=1}^\infty  ip_i(\Delta (\ell_2-\ell_1)) \\
&= \nu^2\Delta^2  \pabbrevvv  e^{\lambda \Delta \ell_2} \leq
C \Delta^2 \theta^k e^{\lambda \Delta \ell_2},
\end{align*}
where in the last step, we use that $ \pabbrevvv  = p_{j,+}^k(t-\Delta\ell_2) + O(\Delta)$ and \eqref{eq:pkbound}.
\end{proof}

\subsection{Proof of Proposition \ref{thm:passageApprox}} \label{app:approxtime}

\begin{proof}
Define
\begin{align*}
    \Omega_\infty^\ast := \left\{ \omega \in  \Omega_\infty   : e^{-\lambda t}Z(t,\omega) \to Y(\omega), Y(\omega)>0\right\}
\end{align*}
It suffices to prove that
\[
\lim_{N \to \infty} |\tau_N-t_N| =0
\]
on the set $\Omega_\infty^\ast$.
Let $\delta>0$ and choose $\varepsilon>0$ small enough so that $e^{\lambda \delta}(1-\varepsilon)>1$ and $e^{-\lambda\delta}(1+\varepsilon)<1$.
On $\Omega_\infty^\ast$, we can find $T = T(\omega)$ such that for all $t \geq T$,
\begin{align*}
    (1-\varepsilon)Y \leq e^{-\lambda t}Z(t) \leq (1+\varepsilon)Y.
\end{align*}
Recall that for $N \geq 1$, $Ye^{\lambda t_N} = N$ i.e.~$t_N = \log(N/Y)/\lambda$ by the definition of $t_N$ in \eqref{eq:tNdef}.
Since $t_N \uparrow \infty$ as $N \to \infty$, we can choose $N = N(\omega)$ such that $t_n \geq t_N > T+\delta$ for all $n \geq N$ and $Z(t) < N$ for all $t < T$.
Then, for $n \geq N$, 
where we use that $Y e^{\lambda t_n} = n$,
\begin{align*}
    Z(t_n+\delta) \geq e^{\lambda(t_n+\delta)}(1-\varepsilon)Y = e^{\lambda \delta}(1-\varepsilon)n > n,
\end{align*}
which implies that $\tau_n \leq t_n+\delta$. Furthermore, for all $n \geq N$ and $T \leq s \leq t_n-\delta$,
\begin{align*}
    Z(s) \leq e^{\lambda s}(1+\varepsilon)Y \leq e^{\lambda(t_n-\delta)}(1+\varepsilon)Y = e^{-\lambda\delta}(1+\varepsilon)n < n,
\end{align*}
which together with $Z(t)<N$ for all $t < T$ implies that $Z(s) < n$ for all $s \leq t_n-\delta$, i.e.~$\tau_n \geq t_n-\delta$. Since $t_n-\delta \leq \tau_n \leq t_n+\delta$ for all $n \geq N$, we can conclude that
\begin{align*}
     \limsup_{n \to \infty} |t_n(\omega)-\tau_n(\omega)| \leq \delta
\end{align*}
for each $\omega \in \Omega_\infty^\ast$. Since $\delta>0$ is arbitrary, we get the result.
\end{proof}

\subsection{Proof of Proposition \ref{thm:totalnummutresult}} \label{app:totalmutgeneral}
\begin{proof}
We use a similar argument to the proof of
    Theorem \ref{thm:mainresult}. First, we break the total number of mutations $M(t)$ into
    \begin{align*}
        M(t)  = M_+(t) - M_-(t),
    \end{align*}
    where $M_+(t)$ represents the total number of mutations generated up until time $t$, and $M_-(t)$ represents the number of mutations which belong to $M_+(t)$ but die out before time $t$. Obviously, these two processes are increasing in time. 
    The limit theorems for $M(t)$ will follow from limit theorems for $M_+(t)$ and $M_-(t)$. Because of the almost identical arguments, we will focus on the analysis of $M_+(t)$.

    As in the proof of Theorem \ref{thm:mainresult}, we define the approximations
    \begin{align}\label{eq:M_approx1}
        \hat{M}_+(t) :=\nu \int_0^t Ye^{\lambda s}ds
    \end{align}
    and 
    \begin{align}\label{eq:M_approx2}
        \bar{M}_+(t) := \nu\int_0^t Z(s)ds,
    \end{align}
    as well as the Riemann sum approximation
    \begin{align}\label{eq:M_approx3}
        \bar{M}_{+,\Delta}(t) := \nu \Delta \sum_{\ell=0}^{\lfloor t/\Delta\rfloor} Z(\ell\Delta) .
    \end{align}
    Note that the only difference between \eqref{eq:approx2} and \eqref{eq:M_approx1} is the probability $p_{j,+}^k(t-s)$ which does not appear in (\ref{eq:M_approx1}).
    Therefore, we can simply follow the proofs of Lemmas \ref{lemma:approx1} and \ref{lemma:approx2} by replacing  $S^k_{j,+}(t)$, $\hat{S}^k_{j,+}(t)$, $\bar{S}^k_{j,+}(t)$, $\bar{S}^k_{j,+,\Delta}(t)$ and $\theta$ with $M_+(t)$, $\hat{M}_+(t)$, $\bar{M}_+(t)$, $\bar{M}_{+,\Delta}(t)$ and $1$, respectively, and we will get 
\label{prop:M_combinedapprox}
\begin{align}
   E\big|M_+(t)- \hat{M}_+(t) \big| = O( t^{1/2} e^{\lambda t/2}),
\end{align}
which implies
\begin{align} \label{eq:M_finiteintegral}
 \int_0^{\infty} e^{-\lambda t} E\big|M_+(t)-\hat{M}_+(t)\big|dt < \infty.
\end{align}
Note that $\lim\limits_{t\to\infty}e^{-\lambda t}\hat{M}_+(t) =\nu Y /\lambda $ exists and $M_+(t)$ is an increasing process. By replacing the corresponding terms in the proof of Proposition \ref{prop:Sj+(t)convergence}, we can get
\begin{align}\label{M+convergence}
    \lim_{t\to\infty} e^{-\lambda t} M_+(t) =\nu Y \int_0^{\infty} e^{-\lambda s}ds = \nu Y /\lambda,
\end{align}
almost surely.
Similarly, 
\begin{align}\label{M-convergence}
    \lim_{t\to\infty} e^{-\lambda t} M_-(t) = \nu Y \int_0^{\infty} e^{-\lambda s}p_0(s)ds,
\end{align}
almost surely.
The fixed-time result \eqref{eq:totalmutfixedtime} follows immediately from (\ref{M+convergence}) and (\ref{M-convergence}).

Then, by following the proof in Section \ref{sec:fixedsizeproof}, we can get the fixed-size result \eqref{eq:totalmutfixedsize} for the total number of mutations,
\begin{align*}
    \lim_{N \to \infty} N^{-1} M(\tau_N) = \nu \int_0^\infty e^{-\lambda s} (1-p_0(s)) ds,
\end{align*}
almost surely.
\end{proof}

\subsection{Proof of Corollary \ref{corollary}} \label{app:bdsimplification}

\begin{proof}
        \begin{enumerate}[(1)]
        \item
    For the birth-death process,
we can write
\begin{align} \label{eq:sizedisgeneral}
\begin{split}
    p_0(t) &= \frac{p(e^{\lambda t}-1)}{e^{\lambda t}-p},  \\
    p_j(t) &= \frac{q^2 e^{\lambda t}}{(e^{\lambda t}-p)^2} \cdot \left(\frac{e^{\lambda t}-1}{e^{\lambda t}-p}\right)^{j-1}, \quad j \geq 1,
\end{split}
\end{align}
see expression (B.1) in \cite{gunnarsson2021exact}.
Therefore, for $j \geq 1$,
\begin{linenomath*}
\begin{align*}
& \int_0^{\infty} e^{-\lambda s} p_j(s) ds = \frac1{\lambda} \int_0^{\infty} \frac{q^2 e^{-\lambda s}}{(1-p e^{-\lambda s})^2}\cdot \left(\frac{1-e^{-\lambda s}}{1-p e^{-\lambda {s}}}\right)^{j-1}  \cdot \lambda e^{-\lambda s} ds.
\end{align*}
\end{linenomath*}
Using the substitution $x := e^{-\lambda s}$, $dx = -\lambda e^{-\lambda s}ds$, we obtain
\[
   \int_0^{\infty} e^{-\lambda s} p_j(s) ds = \frac{q^2}{\lambda}  \int_0^{1} \frac{x}{(1-p x)^2}\cdot \left(\frac{1-x}{1-p x}\right)^{j-1}  dx.
\]
We again change variables, this time $y := (1-x)/(1-p x)$, in which case
\begin{linenomath*}
\begin{align*}
\begin{array}{ll}
     &x = (1-y)/(1-p y), \\
     & dx= -\big(q/(1-p y)^2\big)dy, \\
    &1-p x = q/(1-p y).
\end{array}
\end{align*}
\end{linenomath*}
In addition, $y = 1$ for $x=0$ and $y=0$ for $x=1$, which implies
\begin{align} \label{eq:bdintegral}
    & \int_0^{\infty} e^{-\lambda s} p_j(s) ds = \frac{q}{\lambda}  \int_0^{1} (1-p y)^{-1} (1-y) y^{j-1} dy.
\end{align}
To get the sum representation in \eqref{eq:fixedtimemainresultbd}, it suffices to note that
\begin{linenomath*}
\begin{align*}
   \int_0^{1} (1-py)^{-1} (1-y) y^{j-1} dy &=  \sum_{k=0}^\infty p^k \left( \int_0^{1} (1-y) y^{j+k-1} dy\right) \\
   &=  \sum_{k=0}^\infty \frac{p^k}{(j+k)(j+k+1)}.
\end{align*}
\end{linenomath*}
To get the pure-birth process result, it suffices to note that $p=0$, $q=1$ and
\begin{align*}
    \int_0^1 (1-y) y^{j-1}dy = \frac1{j(j+1)}.
\end{align*}
\item Follows from the same calculations as in (1).
\item By \eqref{eq:sizedisgeneral}, for the birth-death process,
\begin{align*}
    1-p_0(t) = \frac{(1-p)e^{\lambda t}}{e^{\lambda t}-p} = \frac{qe^{\lambda t}}{e^{\lambda t}-p}.
\end{align*}
Therefore,
\begin{linenomath*}
\begin{align*}
& \int_0^{\infty} e^{-\lambda s} (1-p_0(s)) ds = \frac1{\lambda} \int_0^{\infty} \frac{q}{1-pe^{-\lambda s}} \cdot \lambda e^{-\lambda s} ds.
\end{align*}
\end{linenomath*}
Using the substitution $x := e^{-\lambda s}$, $dx = -\lambda e^{-\lambda s}ds$, we obtain
\begin{align} \label{eq:totalmutintegralbd}
\int_0^{\infty} e^{-\lambda s} (1-p_0(s)) ds = \dfrac1{\lambda} \int_0^{1} \frac{q}{1-p x} dx 
= \begin{cases} \dfrac1\lambda, & p=0, \\ - \dfrac{q \log(q)}{\lambda p}, & 0<p<1. \end{cases}
\end{align}
\item Follows from the same calculations as in (3).
\end{enumerate}
\end{proof}

\subsection{Derivation of expression \eqref{eq:proportionofmutbdmorethanj}} \label{app:propfoundinmroethanjcells}

By writing $M_j(t) = M(t) - \sum_{k=0}^{j-1} S_k(t)$, it follows from Corollary \ref{corollary} that conditional on $\Omega_\infty$,
\begin{align*}
    \lim_{t \to \infty} e^{-\lambda t} M_j(t) &= \frac{\nu q Y}{\lambda} \int_0^1 (1-py)^{-1} (1-y) \sum_{k=j}^{\infty} y^{k-1} dy \\
    &= \frac{\nu q Y}{\lambda} \int_0^1 (1-py)^{-1} y^{j-1} dy.
\end{align*}
Similarly,
\begin{align*}
    \lim_{N \to \infty} N^{-1} M_j(\tau_N) 
    &= \frac{\nu q}{\lambda} \int_0^1 (1-py)^{-1} y^{j-1} dy.
\end{align*}
It follows that
\begin{align*}
    \lim_{t \to \infty} \frac{S_j(t)}{M_j(t)} = \lim_{N \to \infty} \frac{S_j(\tau_N)}{M_j(\tau_N)} 
    & = \frac{ \int_0^{1} (1-p y)^{-1} (1-y) y^{j-1}  dy}{\int_0^1 (1-py)^{-1} y^{j-1} dy} \\
    &= 1-\frac{ \int_0^{1} (1-p y)^{-1} y^{j}  dy}{\int_0^1 (1-py)^{-1} y^{j-1} dy} =: \varphi_j(p).
\end{align*}

\subsection{Proof that $\varphi_j(p)$ is strictly decreasing} \label{app:varphidecreasing}

Here, we show that for each $j \geq 1$,  $\varphi_j(p)$ given by the last expression in Section \ref{app:propfoundinmroethanjcells} is strictly decreasing in $p$.
Set
\begin{align*}
    & a := \left( \int_0^{1} (1-p y)^{-2} y^{j+1}  dy\right)\left(\int_0^1 (1-py)^{-1} y^{j-1} dy\right), \\
    & b := \left(\int_0^1 (1-py)^{-2} y^{j} dy\right) \left(\int_0^{1} (1-p y)^{-1} y^{j}  dy\right).
\end{align*}
It suffices to show that $a>b$ for each $p \in (0,1)$.
First, note that we can write
\begin{align*}
    & 
    a = \int_0^{1} \int_0^{1} (1-p y)^{-2} y^{j+1} (1-px)^{-1} x^{j-1} dy dx \\
\end{align*}
and
\begin{align*}
    & 
    b = \int_0^{1} \int_0^{1} (1-py)^{-2} y^{j} (1-px)^{-1} x^{j} dy dx,
\end{align*}
which implies
\begin{align*}
    a-b &= \int_0^1 \int_0^1 (1-py)^{-2}(1-px)^{-1}y^jx^{j-1}(y-x)dydx \\
    &= \int_0^1 \int_0^x (1-py)^{-2}(1-px)^{-1}y^jx^{j-1}(y-x)dydx \\
    &\quad+ \int_0^1 \int_x^1 (1-py)^{-2}(1-px)^{-1}y^jx^{j-1}(y-x)dydx.
\end{align*}
The latter integral can be rewritten as follows:
\begin{align*}
    & \int_0^1 \int_x^1 (1-py)^{-2}(1-px)^{-1}y^jx^{j-1}(y-x)dydx \\
    &= \int_0^1 \int_0^y (1-py)^{-2}(1-px)^{-1}y^jx^{j-1}(y-x)dxdy \\
    &= -\int_0^1 \int_0^x (1-px)^{-2}(1-py)^{-1}x^jy^{j-1}(y-x)dydx
\end{align*}
which implies
\begin{align*}
    a-b = \int_0^1 \int_0^x (1-py)^{-1}(1-px)^{-1}y^{j-1}x^{j-1}(y-x)\big((1-py)^{-1}y-(1-px)^{-1}x\big)dydx.
\end{align*}
Since
\begin{align*}
    \frac{y}{1-py} - \frac{x}{1-px} = \frac{y-x}{(1-py)(1-px)},
\end{align*}
we can finally conclude that
\begin{align*}
    a-b = \int_0^1 \int_0^x (1-py)^{-2}(1-px)^{-2}y^{j-1}x^{j-1}(y-x)^2dydx > 0
\end{align*}
for each $p \in (0,1)$.

\subsection{Derivation of expression \eqref{eq:varphiexpr}} \label{app:estimator}

To derive expression \eqref{eq:varphiexpr} in the main text, we note that $(1-py)^{-1} = \sum_{k=0}^\infty (py)^k$ for $0 < p < 1$ and $0 \leq y \leq 1$, which implies
\begin{align*}
      \int_0^{1} (1-p y)^{-1} (1-y)  dy &= \sum_{k=0}^\infty p^k \int_0^1 y^k(1-y) dy \\
    &=  \sum_{k=0}^\infty \frac{p^k}{k+1} -  \sum_{k=0}^\infty \frac{p^k}{k+2}.
\end{align*}
Since $\sum_{k=1}^\infty \frac{x^k}{k} = -\log(1-x)$, we obtain
\begin{align*}
     \int_0^{1} (1-p y)^{-1} (1-y)  dy 
    &= - \frac{\log(q)}p - \frac{1}{p^2} \big(-\log(q) - p\big) \\
    &= \frac{q}{p^2} \log(q) + \frac{1}{p}.
\end{align*}
Therefore, applying expression \eqref{eq:proportionofmutbd}, we can write for $0<p<1$,
\begin{align*}
&\varphi_1(p) = -\frac{p}{\log(q)} \int_0^{1} (1-p y)^{-1} (1-y)  dy 
= 
-\frac{p+q\log(q)}{p\log(q)}.
\end{align*}

\section*{Acknowledgments}
The authors are indebted to an anonymous reviewer of a previous version of this manuscript, who made several valuable comments and suggested to us the proof of Proposition 1.
EBG was supported in part by NSF grant CMMI-1552764, NIH grant R01 CA241137, funds from the Norwegian Centennial Chair grant and the Doctoral Dissertation Fellowship from the University of Minnesota. K. Leder was supported in part with funds from NSF award CMMI 2228034
and Research Council of Norway Grant 309273.

\section*{Competing Interests Statement}
The authors have no competing interests to declare.

\bibliographystyle{ieeetr}
\bibliography{epi}

\begin{thebibliography}{10}

\bibitem{zeng2006statistical}
K.~Zeng, Y.-X. Fu, S.~Shi, and C.-I. Wu, ``Statistical tests for detecting
  positive selection by utilizing high-frequency variants,'' {\em Genetics},
  vol.~174, no.~3, pp.~1431--1439, 2006.

\bibitem{achaz2009frequency}
G.~Achaz, ``Frequency spectrum neutrality tests: one for all and all for one,''
  {\em Genetics}, vol.~183, no.~1, pp.~249--258, 2009.

\bibitem{sottoriva2015big}
A.~Sottoriva, H.~Kang, Z.~Ma, T.~A. Graham, M.~P. Salomon, J.~Zhao,
  P.~Marjoram, K.~Siegmund, M.~F. Press, D.~Shibata, {\em et~al.}, ``A big bang
  model of human colorectal tumor growth,'' {\em Nat.~Genet.}, vol.~47, no.~3,
  pp.~209--216, 2015.

\bibitem{ling2015extremely}
S.~Ling, Z.~Hu, Z.~Yang, F.~Yang, Y.~Li, P.~Lin, K.~Chen, L.~Dong, L.~Cao,
  Y.~Tao, {\em et~al.}, ``Extremely high genetic diversity in a single tumor
  points to prevalence of non-{D}arwinian cell evolution,'' {\em
  Proc.~Natl.~Acad.~Sci.~USA}, vol.~112, no.~47, pp.~E6496--E6505, 2015.

\bibitem{williams2016identification}
M.~J. Williams, B.~Werner, C.~P. Barnes, T.~A. Graham, and A.~Sottoriva,
  ``Identification of neutral tumor evolution across cancer types,'' {\em
  Nat.~Genet.}, vol.~48, no.~3, p.~238, 2016.

\bibitem{venkatesan2016tumor}
S.~Venkatesan and C.~Swanton, ``Tumor evolutionary principles: how intratumor
  heterogeneity influences cancer treatment and outcome,'' {\em
  Am.~Soc.~Clin.~Oncol.~Educ.~Book}, vol.~36, pp.~e141--e149, 2016.

\bibitem{davis2017tumor}
A.~Davis, R.~Gao, and N.~Navin, ``Tumor evolution: Linear, branching, neutral
  or punctuated?,'' {\em Biochim.~Biophys.~Acta Rev.~Cancer}, vol.~1867, no.~2,
  pp.~151--161, 2017.

\bibitem{durrett2013population}
R.~Durrett, ``Population genetics of neutral mutations in exponentially growing
  cancer cell populations,'' {\em Ann.~Appl.~Prop.}, vol.~23, no.~1, p.~230,
  2013.

\bibitem{durrett2015branching}
R.~Durrett, ``Branching process models of cancer,'' in {\em Branching Process
  Models of Cancer}, pp.~1--63, Springer, 2015.

\bibitem{bozic2016quantifying}
I.~Bozic, J.~M. Gerold, and M.~A. Nowak, ``Quantifying clonal and subclonal
  passenger mutations in cancer evolution,'' {\em PLoS Comput.~Biol.}, vol.~12,
  no.~2, p.~e1004731, 2016.

\bibitem{ohtsuki2017forward}
H.~Ohtsuki and H.~Innan, ``Forward and backward evolutionary processes and
  allele frequency spectrum in a cancer cell population,'' {\em
  Theor.~Popul.~Biol.}, vol.~117, pp.~43--50, 2017.

\bibitem{dinh2020statistical}
K.~N. Dinh, R.~Jaksik, M.~Kimmel, A.~Lambert, S.~Tavar{\'e}, {\em et~al.},
  ``Statistical inference for the evolutionary history of cancer genomes,''
  {\em Stat.~Sci.}, vol.~35, no.~1, pp.~129--144, 2020.

\bibitem{gunnarsson2021exact}
E.~B. Gunnarsson, K.~Leder, and J.~Foo, ``Exact site frequency spectra of
  neutrally evolving tumors: A transition between power laws reveals a
  signature of cell viability,'' {\em Theoretical Population Biology},
  vol.~142, pp.~67--90, 2021.

\bibitem{morison2023single}
C.~Morison, D.~Stark, and W.~Huang, ``Single-cell mutational burden
  distributions in birth-death processes,'' {\em arXiv preprint
  arXiv:2309.06355}, 2023.

\bibitem{tung2021signatures}
H.-R. Tung and R.~Durrett, ``Signatures of neutral evolution in exponentially
  growing tumors: A theoretical perspective,'' {\em PLOS Computational
  Biology}, vol.~17, no.~2, p.~e1008701, 2021.

\bibitem{bonnet2023site}
C.~Bonnet and H.~Leman, ``Site frequency spectrum of a rescued population under
  rare resistant mutations,'' {\em arXiv preprint arXiv:2303.04069}, 2023.

\bibitem{lambert2009allelic}
A.~Lambert, ``The allelic partition for coalescent point processes,'' {\em
  Markov Process.~Relat.~Fields}, vol.~15, no.~3, pp.~359--386, 2009.

\bibitem{lambert2018coalescent}
A.~Lambert, ``The coalescent of a sample from a binary branching process,''
  {\em Theoretical population biology}, vol.~122, pp.~30--35, 2018.

\bibitem{johnston2019genealogy}
S.~G. Johnston, ``The genealogy of {G}alton-{W}atson trees,'' 2019.

\bibitem{harris2020coalescent}
S.~C. Harris, S.~G.~G. Johnston, and M.~I. Roberts, ``{The coalescent structure
  of continuous-time Galton–Watson trees},'' {\em The Annals of Applied
  Probability}, vol.~30, no.~3, pp.~1368 -- 1414, 2020.

\bibitem{johnson2023estimating}
B.~Johnson, Y.~Shuai, J.~Schweinsberg, and K.~Curtius, ``Estimating single cell
  clonal dynamics in human blood using coalescent theory,'' {\em bioRxiv},
  pp.~2023--02, 2023.

\bibitem{schweinsberg2023asymptotics}
J.~Schweinsberg and Y.~Shuai, ``Asymptotics for the site frequency spectrum
  associated with the genealogy of a birth and death process,'' {\em arXiv
  preprint arXiv:2304.13851}, 2023.

\bibitem{durrett2008probability}
R.~Durrett, {\em Probability models for DNA sequence evolution}.
\newblock Springer Science \& Business Media, 2008.

\bibitem{cheek2020genetic}
D.~Cheek and T.~Antal, ``Genetic composition of an exponentially growing cell
  population,'' {\em Stochastic Processes and their Applications}, 2020.

\bibitem{cheek2018mutation}
D.~Cheek and T.~Antal, ``Mutation frequencies in a birth--death branching
  process,'' {\em Ann.~Appl.~Probab.}, vol.~28, no.~6, pp.~3922--3947, 2018.

\bibitem{harris1964theory}
T.~E. Harris, ``The theory of branching process,'' 1964.

\bibitem{athreya2004branching}
K.~B. Athreya and P.~E. Ney, {\em Branching processes}.
\newblock Courier Corporation, 2004.

\bibitem{foo2014escape}
J.~Foo, K.~Leder, and J.~Zhu, ``Escape times for branching processes with
  random mutational fitness effects,'' {\em Stochastic Processes and Their
  Applications}, vol.~124, no.~11, pp.~3661--3697, 2014.

\end{thebibliography}

\end{document}